\documentclass[12pt,reqno]{amsart}
\usepackage[margin=1in,footskip=0.25in]{geometry}
\usepackage{amsmath,amssymb,amsfonts,hyperref,graphicx,tikz,float}
\hypersetup{hidelinks}

\usepackage{enumitem}
\makeatletter
\@namedef{subjclassname@2020}{2020 Mathematics Subject Classification}
\makeatother

\newtheorem{theorem}{Theorem}[section]
\newtheorem{lemma}[theorem]{Lemma}

\newtheorem{definition}[theorem]{Definition}

\newtheorem{corollary}[theorem]{Corollary}

\newtheorem{remark}[theorem]{Remark}

\newtheorem{observation}[theorem]{Observation}

\theoremstyle{definition}
\def\A{\mathcal{A}}

\def \h{\mathcal{H}}
\def \G{\mathcal{G}}
\def \g{\mathcal{G}}
\def \T{\mathcal{T}}

\def\x{\mathbf{x}}%\boldsymbol

\begin{document}

\title[Perturbation of the largest matching root]
{\text{\small Perturbation of the largest matching root of hypergraphs}}

%\date{\today}


\author{Jiang-Chao Wan $^1$, Yi Wang $^2$}
\address{1. School of Mathematics and Statistics, Hefei University, Hefei 230601, Anhui, China \\
	2. School of Mathematical Sciences, Anhui University, Hefei 230601, Anhui, China}
\email{wanjc@stu.ahu.edu.cn, wangy@ahu.edu.cn}
\thanks{{\it Corresponding author.} Yi Wang}
\thanks{{\it Funding.} Supported by the National Natural Science Foundation of China (No. 12331012, 12571360)}

\maketitle

\begin{abstract}
The largest matching root of a $k$-graph is the largest real root of its matching polynomial,
which is equal to the maximum modulus of all the zeros of the matching polynomial.
In this paper, we investigate the perturbation of the largest matching root of $k$-graphs.
We determine all $k$-graphs
whose largest matching root attains the maximum among all $k$-cacti and linear $k$-cacti with a given number of cycles and edges,
where a $k$-cactus is a $k$-graph in which every two distinct cycles have at most one vertex in common.
 To achieve this, we prove that the celebrated shifting operation of $k$-graphs, introduced by Erd\H{o}s, Ko and Rado,
does not decrease the largest matching root.
This result extends a classical result by
Csikv\'ari (Electron. J. Combin. {\bf 18} (2011) $\#$P182)
stating that
the Kelmans transformation does not decrease the largest matching root of graphs.
\end{abstract}

\section{Introduction}\label{section1}
The problem of determining
the graphs for which the largest matching root achieves the maximum
in a given family of graphs
can be traced back to the celebrated paper by Lov\'asz and Pelik\'an~\cite{LovaszPe} in 1973.
They~\cite{LovaszPe} %$,  Lov\'asz and Pelik\'an
 determined the first two largest spectral radii of trees with a given number of vertices.
Since then, many researchers studied the problem of determining
the graphs for which the largest matching root achieves the maximum in a given family of graphs, see, e.g.~\cite{Liu,ZhangCY,ZhangLY}.
In this paper, we extend the problem to uniform hypergraphs.
Before stating the main results of this paper,
we first introduce the necessary background on the matching polynomials of graphs and uniform hypergraphs.

\subsection{Background}
A {\it $k$-uniform hypergraph} (or {\it $k$-graph} for short) $\h=(V(\h), E(\h))$ consists of a vertex set $V(\h)$ and an edge set $E(\h)$,
where each edge $e\in E(\h)$ is a $k$-element subset of $V(\h)$.
Clearly, $2$-graph are the ordinary graph.
A {\em matching} in $\h$ is a set of vertex-disjoint edges, and an $r$-matching is a matching having $r$ edges.
To study the spectral radius of the adjacency tensor of $k$-trees, Su et al.~\cite{Suejc} introduced the {\em matching polynomial} of a $k$-graph $\h$ as follows:
\begin{equation}\label{definitionmp}
\mu(\h,x)=\sum_{r\geq 0}(-1)^rp(\h,r) x^{|V(\h)|-kr},
\end{equation}
where $p(\h,r)$ denotes the number of $r$-matchings of $\h$.
%, with the convention that $p(\h,0)=1$.
The case of $k=2$ in~\eqref{definitionmp} is the classical matching polynomial of a graph introduced by Heilmann and Lieb~\cite{Heilmann1,Heilmann}.
We refer the reader to the textbook~\cite{Godsil} for more history on the matching polynomial theory.

Denote by $\lambda(\h)$ the maximum modulus of all the zeros of the matching polynomial of a $k$-graph $\h$.
The Heilmann--Lieb theorem says that the zeros of the matching polynomial of a graph are real and symmetrically distributed about the origin.
Thus, $\lambda(G)$ is equal to the largest root of $\mu(G, x)$ for a graph $G$.
However, the case of $k\geq 3$ seems to be a little more complex since the matching polynomial of a $k$-graph with $k\geq 3$ must contain a non-real zero \cite{WanWF}.
Zhang et al.~\cite{Zhang} proved that for a $k$-tree $\T$ the largest root of $\mu(\T,x)$ is equal to the spectral radius of the adjacency tensor of $\T$.
Based on this fact and some properties of the matching polynomial,
Su et al.~\cite{Suejc}
determined the first $\lfloor d/2\rfloor+1$ largest spectral radii of the adjacency tensor of $k$-trees with size $m$ and diameter $d$.

Recently, the authors \cite{WanWF} established a hypergraph Heilmann--Lieb theorem.
In particular, they showed that for a $k$-graph $\h$, $\lambda(\h)$ is a simple root of $\mu(\h,x)$,
and hence that $\lambda(\h)$ is the largest root of $\mu(\h,x)$ (see \cite[Theorem 1.2]{WanWF} for details).
We simply call $\lambda(\h)$ the {\it largest matching root} of $\h$.
A natural question is to determine the $k$-graphs for which the largest matching root achieves the maximum in a given family of $k$-graphs.

\subsection{Main results}\label{Mainsubsection}

Let $\mathcal{P} = (v_0, e_1, v_1, \dots, e_\ell, v_\ell)$
be an alternating sequence of vertices and edges in $\h$ such that $e_1, \ldots, e_\ell$ are pairwise distinct and $v_{i-1}, v_i\in e_i$ for all $i =1, \ldots, \ell$.
The sequence $\mathcal{P}$ is called a {\it path} in $\h$
if the vertices $v_0, v_1, \ldots, v_\ell$ are pairwise distinct,
and it is called a {\it cycle} in $\h$ if $v_1, \ldots, v_\ell$ are pairwise distinct, $\ell\geq 2$, and $v_0= v_\ell$.
Denote by $V(\mathcal{P})$ the set $\cup_{i=1}^\ell e_i$.
A connected $k$-graph $\h$ is called a {\it $k$-cactus} if every two distinct cycles of $\h$ have at most one vertex in common,
that is, $|V(\mathcal{C}_1) \cap V(\mathcal{C}_2)|\leq 1$ for  every two distinct cycles $\mathcal{C}_1$ and $\mathcal{C}_2$ of $\h$.
For integers $k\geq 3$ {\rm(}resp. $k\geq 2${\rm)}, $m \geq 1$ and $t\geq 0$,
we denote by $\mathfrak{Ca}_k(m,t)$ {\rm(}resp. $\mathfrak{LCa}_k(m,t)${\rm)} the set of all {\rm(}resp. linear{\rm)} $k$-cacti with $m$ edges and $t$ cycles.
Here, a $k$-graph $\h$ is said to be {\it linear} if any two distinct edges of $\h$ intersect at most one vertex.

For a given family $\mathfrak{F}$ of $k$-graphs,
a $k$-graph $\h\in\mathfrak{F}$ is called {\it maximizing} in $\mathfrak{F}$ if $\lambda(\h)$ attains the maximum among all $k$-graphs in $\mathfrak{F}$.
The aim of this paper is to determine all maximizing $k$-graphs in $\mathfrak{Ca}_k(m,t)$ and $\mathfrak{LCa}_k(m,t)$,
respectively.

\begin{figure}[H]
	\begin{tikzpicture}[scale=0.68]
	\begin{scope}[shift={(-16cm,0)}]
	%%%%%%%%%%%%%%%%%%%%%%%%%%%%%%%%%%%%%%%%%%%%%%%%%%%%%% %%%%%%%%%%%%%%%%%%%%%%%%%%%%%%%%%%%%%%%%%%%%%%%%%%%%%%
	
	\filldraw[fill=black,draw=black] (12, 0)  circle [radius=0.1] ;

	\filldraw[fill=black,draw=black] (10 , 2.5)  circle [radius=0.07];
	
	\filldraw[fill=black,draw=black] (14, 2.5)  circle [radius=0.07];

	%%%%%%%%%%%%
	\draw[thick] (11.1,-1.23) ellipse[x radius=0.2,y radius=2,rotate = -36];

	\draw[thick] (11.6,-1.5) ellipse[x radius=0.2,y radius=2,rotate = -16];
	
	\draw[thick] (12.9,-1.23) ellipse[x radius=0.2,y radius=2,rotate = 36];

	%%%%%%%%%%%
	\filldraw[fill=black,draw=black] (11.3,     -1)  circle [radius=0.07];
	\filldraw[fill=black,draw=black] (11.75,     -1)  circle [radius=0.07];
	\filldraw[fill=black,draw=black] (12.7,     -1)  circle [radius=0.07];

	%%%%%%%%%%%
	\filldraw[fill=black,draw=black] (10.35,     -2.3)  circle [radius=0.07];
	\filldraw[fill=black,draw=black] (11.35,     -2.4)  circle [radius=0.07];
	\filldraw[fill=black,draw=black] (13.65,     -2.3)  circle [radius=0.07];

	%%%%%%%%%%%%%%%%%%%%
	\filldraw[fill=black,draw=black] (11,  -1.4)  circle [radius=0.03];
	\filldraw[fill=black,draw=black] (10.85,  -1.6)  circle [radius=0.03];
	\filldraw[fill=black,draw=black] (10.7,  -1.8)  circle [radius=0.03];

	%%%%%%%%%%%%%%%%%%%%
	\filldraw[fill=black,draw=black] (11.6,  -1.5)  circle [radius=0.03];
	\filldraw[fill=black,draw=black] (11.55,  -1.7)  circle [radius=0.03];
	\filldraw[fill=black,draw=black] (11.5,  -1.9)  circle [radius=0.03];

	%%%%%%%%%%%%%%%%%%%%
	\filldraw[fill=black,draw=black] (13,  -1.4)  circle [radius=0.03];
	\filldraw[fill=black,draw=black] (13.15,  -1.6)  circle [radius=0.03];
	\filldraw[fill=black,draw=black] (13.3,  -1.8)  circle [radius=0.03];

	%%%%%%%%%%%%%%%%%%%%
	\filldraw[fill=black,draw=black] (12,  -2.6)  circle [radius=0.03];
	\filldraw[fill=black,draw=black] (12.55,  -2.6)  circle [radius=0.03];
	\filldraw[fill=black,draw=black] (13.1,  -2.6)  circle [radius=0.03];

	\draw[<->,thick , dotted] (9.6,  -3.3) to [bend right] (14.4,-3.1);
	\node at (12,-4.5) {$m-2t$ edges};
	\node at (12,-5.2) {$\h^{(k)}_{m,t}$};
	
	\draw[<->,thick , dotted] (9.6,   3 ) to [bend left] (14.4, 3 );
	\node at (12,4.1) {$t$ cycles of length $2$};			
	
	\draw[thick,draw=red] (12.2,-0.2) to[out=60,in=-90] (12, 1.3) to[out=90,in=110](9.8, 2.5) to[out=-70,in=110] (11.5, 1.3) to[out=-70,in=-140] (12.2,-0.2);
	
	\draw[thick,draw=blue] (12.2, 0) to[out=-60,in=-70] (10, 1 )  to[out=110,in=-120](9.9, 2.7) to[out=60,in=120] (10.5, 1.3) to[out=-60 ,in=120] (12.2,0) ;
	
	\draw[thick,draw=green] (12 ,-0.2) to [out=-180,in=-130 ] (12.4, 1.5)to [out=50,in= 80 ] (14.2,2.4) to [out=-100,in=70 ] (13, 1.5) to [out=-110,in= 50 ] (12 ,-0.2);

	\draw[thick,draw=pink] (11.9,-0.2) to [out=10,in=-100] (14 , 1.5) to[out=80 ,in=30] (13.9,2.7)to [out=-150,in=90] (13.5 , 1.5)to [out=-90,in=170] (11.9,-0.2);

	\filldraw[fill=black,draw=black] (11 , 0.5)  circle [radius=0.07];
	\filldraw[fill=black,draw=black] (10.2, 1.2)  circle [radius=0.03];
	\filldraw[fill=black,draw=black] (10.1, 1.6)  circle [radius=0.03];
	\filldraw[fill=black,draw=black] (10.5, 0.9)  circle [radius=0.03];

	\filldraw[fill=black,draw=black] (11.8, 1)  circle [radius=0.07];
	
	\filldraw[fill=black,draw=black] (10.9, 2.1)  circle [radius=0.03];
	\filldraw[fill=black,draw=black] (11.35, 1.85)  circle [radius=0.03];
	\filldraw[fill=black,draw=black] (11.65, 1.5)  circle [radius=0.03];

	\filldraw[fill=black,draw=black] (12.5, 1)  circle [radius=0.07];

	\filldraw[fill=black,draw=black] (13.35, 2.05)  circle [radius=0.03];
	\filldraw[fill=black,draw=black] (13, 1.85)  circle [radius=0.03];
	\filldraw[fill=black,draw=black] (12.75, 1.5)  circle [radius=0.03];

	\filldraw[fill=black,draw=black] (13.2, 0.6)  circle [radius=0.07];

	\filldraw[fill=black,draw=black] (13.8, 1.5)  circle [radius=0.03];
	\filldraw[fill=black,draw=black] (13.7, 1.2)  circle [radius=0.03];
	\filldraw[fill=black,draw=black] (13.55, 0.95)  circle [radius=0.03];
	
	\end{scope}

	%%%%%%%%%%%%%%%%%%%%%%%%%%%%%%%%%%%%%%%%%%%%%%%%% %%%%%%%%%%%%%%%%%%%%%%%%%%%%%%%%%%%%%%%%%%%%%%%%%
	\begin{scope}%[shift={(6cm,0)}]
	
	\draw[thick] (4.6,1 ) ellipse[x radius=0.2,y radius=1.7,rotate = -30];
	
	\draw[thick] (3.4 ,1 ) ellipse[x radius=0.2,y radius=1.7,rotate = 30];

	%%%%%%%%%%%%%%%%%%%%%%%%%%%%%%%%%%%%%%%%%%%%%%%%%%%%%%%%%%%%%
	%%%%%%%%%%%%%%%%%%%%%%%%%%%%%%%%%%%%%%%%%%%%%%%%%%%%%%%%%%%%%
	
	\draw[thick] (3 ,0.45) ellipse[x radius=0.2,y radius=1.7,rotate = 66];
	\draw[thick] (5 ,0.45) ellipse[x radius=0.2,y radius=1.7,rotate = -66];

	\filldraw[fill=black,draw=black] (3.3, 0.3)  circle [radius=0.07];
	\filldraw[fill=black,draw=black] (4.7, 0.3)  circle [radius=0.07];

	\filldraw[fill=black,draw=black] (2, 0.9)  circle [radius=0.07];
	\filldraw[fill=black,draw=black] (6, 0.9)  circle [radius=0.07];
	
	\filldraw[fill=black,draw=black] (2.4,   0.7)  circle [radius=0.03];
	\filldraw[fill=black,draw=black] (5.6,   0.7)  circle [radius=0.03];

	\filldraw[fill=black,draw=black] (2.65,   0.6)  circle [radius=0.03];
	\filldraw[fill=black,draw=black] (5.35,   0.6)  circle [radius=0.03];

	\filldraw[fill=black,draw=black] (2.9,   0.5)  circle [radius=0.03];
	\filldraw[fill=black,draw=black] (5.1,   0.5)  circle [radius=0.03];
	
	%%%%%%%%%%%%%%%%%%%%%%%%%%%%%%%%%%%%%%%%%%%%%%%%%%%%%%%%%

	\draw[thick] (2.4,1.4) ellipse[x radius=0.2,y radius=1.4,rotate = -38];
	
	\draw[thick] (5.6,1.4) ellipse[x radius=0.2,y radius=1.4,rotate = 38];

	\filldraw[fill=black,draw=black] (2.42,   1.45)  circle [radius=0.03];
	\filldraw[fill=black,draw=black] (5.58,   1.45)  circle [radius=0.03];

	\filldraw[fill=black,draw=black] (2.55,   1.6)  circle [radius=0.03];
	\filldraw[fill=black,draw=black] (5.45,   1.6)  circle [radius=0.03];

	\filldraw[fill=black,draw=black] (2.3,   1.3)  circle [radius=0.03];
	\filldraw[fill=black,draw=black] (5.7,   1.3)  circle [radius=0.03];
	%%%%%%%%%%%%%%%%%%%%%%%%%%%%%%%%%%%%%%%%%%%%%%%%%%%%%%%%
	
	\filldraw[fill=black,draw=black] (2.84, 2)  circle [radius=0.07];
	
	\filldraw[fill=black,draw=black] (5.16, 2)  circle [radius=0.07];
	
	\filldraw[fill=black,draw=black] (3.6, 0.7)  circle [radius=0.07];
	
	\filldraw[fill=black,draw=black] (4.4, 0.7)  circle [radius=0.07];

	%%%%%%%%%%%%%%%%%%%%
	\filldraw[fill=black,draw=black] (3.4,   1)  circle [radius=0.03];
	\filldraw[fill=black,draw=black] (4.6,   1)  circle [radius=0.03];
	
	\filldraw[fill=black,draw=black] (3.28,   1.2)  circle [radius=0.03];
	\filldraw[fill=black,draw=black] (3.16,   1.4)  circle [radius=0.03];

	\filldraw[fill=black,draw=black] (4.73,   1.2)  circle [radius=0.03];
	\filldraw[fill=black,draw=black] (4.84,   1.4)  circle [radius=0.03];

	%%%%%%%%%%%%
	\filldraw[fill=black,draw=black](4, 0) circle [radius=0.1];

	%%%%%%%%%%%%%%%%%%%
	\filldraw[fill=black,draw=black] (3.6,   2)  circle [radius=0.03];
	\filldraw[fill=black,draw=black] (4,   2)  circle [radius=0.03];
	\filldraw[fill=black,draw=black] (4.4,  2)  circle [radius=0.03];
	
	%%%%%%%%%%%%%%%%%%%%%%%%%%%%%%%%%%%%%%%%%%%%%%%%%%%%%%%%%
	\draw[thick] (3.1,-1.23) ellipse[x radius=0.2,y radius=2,rotate = -36];
	
	\draw[thick] (3.6,-1.5) ellipse[x radius=0.2,y radius=2,rotate = -16];
	
	\draw[thick] (4.9,-1.23) ellipse[x radius=0.2,y radius=2,rotate = 36];

	%%%%%%%%%%%%
	\filldraw[fill=black,draw=black] (3.3,     -1)  circle [radius=0.07];
	\filldraw[fill=black,draw=black] (3.75,     -1)  circle [radius=0.07];
	\filldraw[fill=black,draw=black] (4.7,     -1)  circle [radius=0.07];

	%%%%%%%%%%%
	\filldraw[fill=black,draw=black] (2.35,     -2.3)  circle [radius=0.07];
	\filldraw[fill=black,draw=black] (3.35,     -2.4)  circle [radius=0.07];
	\filldraw[fill=black,draw=black] (5.65,     -2.3)  circle [radius=0.07];

	%%%%%%%%%%%%%%%%%%%
	\filldraw[fill=black,draw=black] (3,  -1.4)  circle [radius=0.03];
	\filldraw[fill=black,draw=black] (2.85,  -1.6)  circle [radius=0.03];
	\filldraw[fill=black,draw=black] (2.7,  -1.8)  circle [radius=0.03];

	%%%%%%%%%%%%%%%%%%%%%%%%%%%%%%%%
	\filldraw[fill=black,draw=black] (3.6,  -1.5)  circle [radius=0.03];
	\filldraw[fill=black,draw=black] (3.55,  -1.7)  circle [radius=0.03];
	\filldraw[fill=black,draw=black] (3.5,  -1.9)  circle [radius=0.03];

	%%%%%%%%%%%%%%%%%%%%di san bian xu xian%%%%%%%%%%%%%%%%%%%%%%%%%%%%%%%%
	\filldraw[fill=black,draw=black] (5,-1.4)circle[radius=0.03];
	\filldraw[fill=black,draw=black] (5.15,-1.6)circle[radius=0.03];
	\filldraw[fill=black,draw=black] (5.3,-1.8)circle[radius=0.03];

	\filldraw[fill=black,draw=black] (4,  -2.6)  circle [radius=0.03];
	\filldraw[fill=black,draw=black] (4.55,  -2.6)  circle [radius=0.03];
	\filldraw[fill=black,draw=black] (5.1,  -2.6)  circle [radius=0.03];

	\draw[<->,thick,dotted](1.6,-3.3) to [bend right](6.4,-3.1);
	\node at (4,-4.5) {$m-3t$ edges};
	\node at (4,-5.1) {$\mathcal{\mathcal{L}}^{(k)}_{m,t}$};

	\draw[<->,thick, dotted](1.6,2.3) to [bend left](6.4,2.3);
	\node at (4, 3.4) {$t$ cycles of length $3$};
	\end{scope}
	%%%%%%%%%%%% shang liang dian %%%%%%%%%%%%
	%\filldraw[fill=black,draw=black] (2 , 2.5)  circle [radius=0.07];
	%\filldraw[fill=black,draw=black] (6, 2.5)  circle [radius=0.07];
	\end{tikzpicture}
	\caption{The structures of $\h^{(k)}_{m,t}$ and $\mathcal{\mathcal{L}}^{(k)}_{m,t}$.}\label{figure3}
\end{figure}
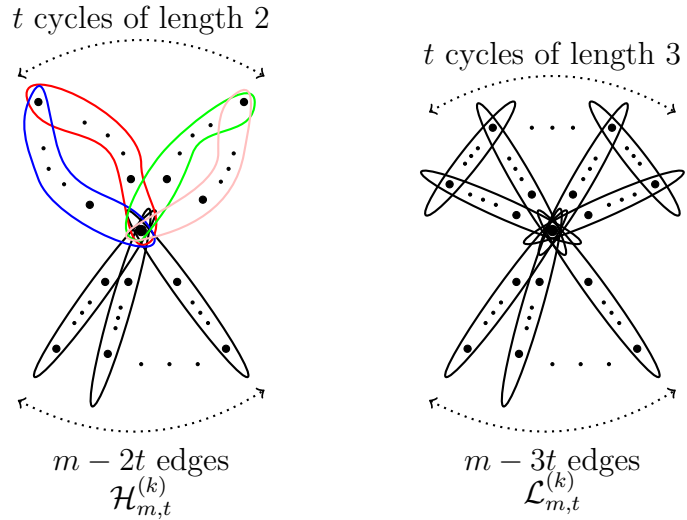

We are now ready to state the main result of this paper.

\begin{theorem}\label{MainresutCactus}
Let $k\geq 3$ {\rm(}resp. $k\geq 2${\rm)}, $m\geq 1$, and $t\geq 0$ be integers.
Let $\h^{(k)}_{m,t}$
{\rm(}resp. $\mathcal{L}^{(k)}_{m,t}${\rm)} be the $k$-graph illustrated in Figure~\ref{figure3}.
Then the following statements hold.
\begin{enumerate}[label={\rm (\alph*)}]% (\roman*)[label=\arabic*)]
		
\item For $(m,t)\neq (3,1)$, $\h^{(k)}_{m,t}$ is the unique maximizing $k$-graph in $\mathfrak{Ca}_k(m,t)$.

\item For $(m,t)=(3,1)$, $\h^{(k)}_{3,1}$ and $\mathcal{L}^{(k)}_{3,1}$ are all maximizing $k$-graphs in $\mathfrak{Ca}_k(m,t)$.

\item The $k$-graph $\mathcal{L}^{(k)}_{m,t}$ is the unique maximizing $k$-graph in $\mathfrak{LCa}_k(m,t)$.
		
\end{enumerate}
\end{theorem}

%The following result is an immediate consequence of Theorem~\ref{MainresutCactus}.

%\begin{corollary}\label{unicyclicMaxM}
%Let $k\geq 3$ {\rm(}resp.~$k\geq 2${\rm)} be an integer.
%The $k$-graph $\h^{(k)}_{m,1}$
%{\rm(}resp. $\mathcal{L}^{(k)}_{m,1}${\rm)}
% is the unique $k$-graph
%whose largest matching root attains the maximum
%among all unicyclic
% {\rm(}resp. linear{\rm)}
%$k$-graphs with $m$ edges.
%\end{corollary}

To achieve Theorem~\ref{MainresutCactus}, we prove that the celebrated shifting operation of $k$-graphs, introduced by Erd\H{o}s, Ko and Rado~\cite{Erdos},
does not decrease the largest matching root.

%The shifting operation on $k$-graphs
\subsection{Main tool: the shifting operation}
The Kelmans transformation is an important tool in algebraic graph theory,
which was introduced by Kelmans~\cite{Kelmans} in 1981 to study the extremal problems related to the synthesis of reliable networks.
Denote by $E_\h(v)$ the set of all edges of the $k$-graph $\h$ containing $v$.
Given a graph $G$ with two vertices $u$ and $v$, the {\it Kelmans transformation} of $G$ from $v$ to $u$
is the operation that removes all edges between $v$ and $N_G(v)\setminus ( N_G(u) \cup \{u\})$ and adds all edges between $u$ and $N_G(v)\setminus ( N_G(u) \cup \{u\})$.
See Figure~\ref{KelmansTfig} for an example.

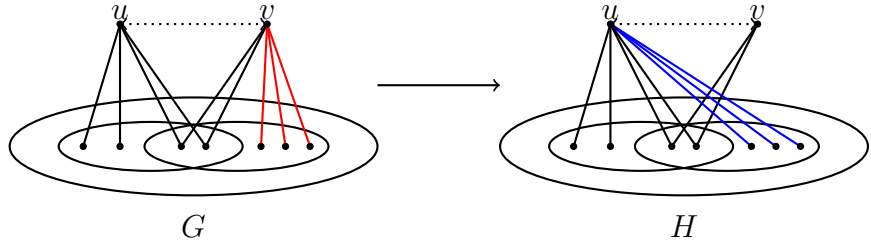
\begin{figure}[H]
\centering
\begin{tikzpicture}[scale=0.81]
		
%%%%%%%%%%%% xia san bian %%%%%%%%%%%%
\draw[thick] (4,0) ellipse[x radius=3 ,y radius=0.8];
		
\draw[thick] (3.3,0) ellipse[x radius=1.5 ,y radius=0.4];
		
\draw[thick] (4.7,0) ellipse[x radius=1.5 ,y radius=0.4];
		
%%%%%%%%%%%% He xin dian %%%%%%%%%%%%
		
\node at (2.8,2.2) {$u$};
		
\node at (5.2,2.2) {$v$};
		
\filldraw[fill=black,draw=black](2.2,0)circle[radius=0.05];
		
\filldraw[fill=black,draw=black](2.8,0)circle[radius=0.05];
		
\draw[ thick] (2.8, 2) -- (2.2, 0);
\draw[ thick] (2.8, 2) -- (2.8, 0);	
		
\filldraw[fill=black,draw=black](4.2,0)circle[radius=0.05];
		
\filldraw[fill=black,draw=black](3.8,0)circle[radius=0.05];
		
\draw[thick] (2.8, 2) -- (4.2, 0);
\draw[thick] (2.8, 2) -- (3.8, 0);
		
\draw[thick] (5.2, 2) -- (4.2, 0);
\draw[thick] (5.2, 2) -- (3.8, 0);

\draw[thick, draw=red] (5.2, 2) -- (5.1, 0);
\draw[thick, draw=red] (5.2, 2) -- (5.5, 0);
\draw[thick, draw=red] (5.2, 2) -- (5.9, 0);
		
\filldraw[fill=black,draw=black](5.1,0)circle[radius=0.05];
		
\filldraw[fill=black,draw=black](5.5,0)circle[radius=0.05];
		
\filldraw[fill=black,draw=black](5.9,0)circle[radius=0.05];

\filldraw[fill=black,draw=black](2.8,2)circle[radius=0.05];
\filldraw[fill=black,draw=black](5.2,2)circle[radius=0.05];
		
\draw[<->,thick , dotted] (2.8, 2) to (5.2, 2);
		
		%%%%%%%%%%%%%%%%%%%%%%%%%%%%%%%%%%%%%%%%%%%%%%%%%%%%%%%%%%%%%%%%%%%%%%%		%%%%%%%%%%%%%%%%%%%%%%%%%%%%%%%%%%%%%%%%%%%%%%%%%%%%%%%%%%%%%%%%%%%%%%%

\draw[thick] (12,0) ellipse[x radius=3 ,y radius=0.8];
		
\draw[thick] (11.3,0) ellipse[x radius=1.5 ,y radius=0.4];
		
\draw[thick] (12.7,0) ellipse[x radius=1.5 ,y radius=0.4];
		
%%%%%%%%%%%% He xin dian %%%%%%%%%%%%

\node at (10.8,2.2) {$u$};
		
\node at (13.2,2.2) {$v$};
		
\filldraw[fill=black,draw=black](10.2,0)circle[radius=0.05];
		
\filldraw[fill=black,draw=black](10.8,0)circle[radius=0.05];
		
\draw[thick] (10.8,2) -- (10.2, 0);
\draw[thick] (10.8,2) -- (10.8, 0);

\filldraw[fill=black,draw=black](12.2,0)circle[radius=0.05];
		
\filldraw[fill=black,draw=black](11.8,0)circle[radius=0.05];
	
\draw[thick] (10.8, 2) -- (12.2, 0);
\draw[thick] (10.8, 2) -- (11.8, 0);
		
\draw[thick] (13.2, 2) -- (12.2, 0);
\draw[thick] (13.2, 2) -- (11.8, 0);

\draw[thick, draw=blue] (10.8, 2) -- (13.1, 0);
\draw[thick, draw=blue] (10.8, 2) -- (13.5, 0);
\draw[thick, draw=blue] (10.8, 2) -- (13.9, 0);
		
\filldraw[fill=black,draw=black](13.1,0)circle[radius=0.05];
		
\filldraw[fill=black,draw=black](13.5,0)circle[radius=0.05];
		
\filldraw[fill=black,draw=black](13.9,0)circle[radius=0.05];

\filldraw[fill=black,draw=black](10.8,2)circle[radius=0.05];
\filldraw[fill=black,draw=black](13.2,2)circle[radius=0.05];
		
\draw[<->,thick , dotted] (10.8, 2) to (13.2, 2);
		
\draw[->,thick ] (7,1)to (9,1);
		
\node at (4, -1.3) {$G$};
		
\node at (12, -1.3) {$H$};
		
\end{tikzpicture}
\caption{An example for Kelmans transformation from $v$ to $u$ in $G$.}\label{KelmansTfig}
\end{figure}

In recent years, the Kelmans transformation has found several applications, such as in the theory of reliable networks~\cite{Brown,Satyanarayana},
spectral graph theory~\cite{CsikvariDM}, and
graph polynomial theory~\cite{Csikvari}.
Moreover, Csikv\'ari~\cite{Csikvari} proved that
the Kelmans transformation does not decrease the largest matching root, and Chen, Li, and Lian~\cite{Chen} subsequently determined the case of equality.
We summarize these two results as the following theorem and refer the reader to the original articles~\cite{Csikvari,Chen} for more details.

\begin{theorem}[\cite{Csikvari,Chen}]\label{KelmTransgraph}
Let $G$ be a connected graph with the largest matching root $\lambda(G)$.
If $\widehat{G}$ is the graph obtained from $G$ by a Kelmans transformation,
then $\lambda(G)\leq \lambda(\widehat{G})$,
with equality if and only if $G$ is isomorphic to $\widehat{G}$.
 \end{theorem}

Using Theorem~\ref{KelmTransgraph},
Chen, Li, and Lian~\cite{Chen} solved a conjecture on the maximum skew-spectral radius of odd-cycle graphs,
and Liu et al.~\cite{Liu} determined the unicyclic graphs with the four largest matching roots.

In 1961,
Erd\H{o}s, Ko and Rado~\cite{Erdos} invented the shifting operation to prove the celebrated Erd\H{o}s--Ko--Rado theorem.
The shifting operation has become a powerful
technique in extremal combinatorics, and we refer the reader to the survey~\cite{Frankl} by Frankl for more details and applications.

%In their celebrated article\footnote{They proved their result in 1938, but did not publish it until 1961.}\,\cite{Erdos},

\begin{definition}[\cite{Erdos}]
Let $\h$ be a $k$-graph with two vertices $u$ and $v$.
For an edge $e\in E(\h)$, its $(u,v)$-shift $\mathbb{S}_{uv}(e)$ is defined by
$$\mathbb{S}_{uv}(e)=\left\{
\begin{array}{ll}
(e\setminus\{v\})\cup\{u\}, & v\in e, u\notin e \text{ and } (e\setminus\{v\})\cup\{u\}\notin E(\h); \\
e, & \hbox{otherwise.}
\end{array}
\right.
$$
The $(u,v)$-shift $\mathbb{S}_{uv}(\h)$ of $\h$ is the $k$-graph
$$\big(V(\h),\left\{\mathbb{S}_{uv}(e):  e\in E(\h) \right\}\big).$$
\end{definition}

We note that for $k=2$, the shifting operation coincides with the Kelmans transformation.
The following result extends Theorem~\ref{KelmTransgraph} from graphs to $k$-graphs, which will be employed
to prove Theorem~\ref{MainresutCactus}.
For an edge  $e\in E(\h)$, we write $\h-e$ for the $k$-graph with $V(\h-e)=V(\h)\setminus e$ and $E(\h-e)=\{f\in E(\h):f \subseteq V(\h)\setminus e\}$.

\begin{theorem}\label{KelmTranshypergraph}
Let $\mathcal{H}$ be a connected $k$-graph with two vertices $u$ and $v$.
Then
$$%\begin{equation}\label{KelmTranshypergraphineq}
\lambda(\mathcal{H})\leq \lambda(\mathbb{S}_{uv}(\h)).
$$%\end{equation}
Moreover, if $\mathbb{S}_{uv}(\h)$ is not isomorphic to $\mathcal{H}$ and there exists an edge $e\in E_\h(u)\setminus E_\h(v)$ such that $\mathbb{S}_{uv}(\h)-e$ is a proper subgraph of $\h-e$, then
$\lambda(\mathcal{H})< \lambda(\mathbb{S}_{uv}(\h)).$
\end{theorem}

\begin{remark}%\label{KelmTranshypergraphremark}
We make the following two remarks for Theorem~\ref{KelmTranshypergraph}.
\begin{enumerate}[label={\rm (\alph*)}]% (\roman*)[label=\arabic*)]
		
\item If $\h$ is not isomorphic to $\mathbb{S}_{uv}(\h)$, then there exist edges $e\in E_\h(u)\setminus E_\h(v)$ and $f\in E_\h(v)\setminus E_\h(u)$ satisfying $\mathbb{S}_{uv}(f)\notin E(\h)$.
    Using this fact, for $k=2$, one may check that $\mathbb{S}_{uv}(\h)-e$ is a proper subgraph of $\h-e$.
    By Theorem~\ref{KelmTranshypergraph}, we get
    $\lambda(\mathcal{H})< \lambda(\mathbb{S}_{uv}(\h)).$
    Hence, Theorem~\ref{KelmTranshypergraph} and  Theorem~\ref{KelmTransgraph} are identical when $k=2$.

\item Comparing  Theorem~\ref{KelmTranshypergraph} and  Theorem~\ref{KelmTransgraph}, for $k\geq 3$,
one may ask whether
$\lambda(\mathcal{H})=\lambda(\mathbb{S}_{uv}(\h))$
if and only if $\h$ is isomorphic to $\mathbb{S}_{uv}(\h)$.
%Unfortunately, t
The answer is `no' since we can construct infinitely many $k$-graphs $\h$ such that $\h$ is not isomorphic to $\mathbb{S}_{uv}(\h)$ but $\lambda(\mathcal{H})=\lambda(\mathbb{S}_{uv}(\h))$ as follows.
Let $\g$ be the $3$-graph with vertex set $\{1,2,3,4,5\}$ and edge set
\[
E(\g)=\{\{1,2,3\},\{1,2,4\},\{1,3,5\}\}.
\]
Then
\[
E(\mathbb{S}_{23}(\g))
=
\{\{1,2,3\},\{1,2,4\},\{1,2,5\}\}.
\]
Observe that $\g$ and $\mathbb{S}_{23}(\g)$ are not isomorphic, but they have the same matching polynomial
\[
\mu(\g,x)=\mu(\mathbb{S}_{23}(\g),x)=x^5-3x^2.
\]
So they have the same largest matching root.
Moreover, for $k\geq 3$, denote by $\g^{(k)}_m$ the $k$-graph obtained from $\g$ by adding $k-3$ new vertices to each edge of $\g$ and attaching $m$ pendant edges to vertex $1$.
One may find that $\g^{(k)}_m$ and $\mathbb{S}_{23}(\g^{(k)}_m)$ are not isomorphic, but they have the same matching polynomial.
%\begin{figure}[H]
%\centering
%\begin{tikzpicture}[scale=0.81, every %node/.style={font=\small}]

% 顶点坐标
%\coordinate (v1) at (0, 2);
%\coordinate (v2) at (-1.5, 0);
%\coordinate (v3) at (1.5, 0);
%\coordinate (v4) at (-0.8, -1.5);
%\coordinate (v5) at (0.8, -1.5);

% 超边 {1,2,3} ― 蓝色椭圆（放大）
%\draw[blue, thick] (0, 0.7) ellipse (2.1cm and 1.7cm);

% 超边 {1,2,4} ― 红色椭圆（放大，顺时针旋转15度）
%\draw[red, thick, rotate around={-15:(-0.5,0.2)}] (-0.5, 0.2) %ellipse (1.3cm and 2.3cm);

% 超边 {1,3,5} ― 绿色椭圆（放大，逆时针旋转15度）
%\draw[green!60!black, thick, rotate around={15:(0.5,0.2)}] %(0.5, 0.2) ellipse (1.3cm and 2.3cm);

% 顶点：黑色实心圆点 + 编号
%\foreach \i/\pos in {1/above, 2/left, 3/right, 4/below left, %5/below right} {
 %   \fill (v\i) circle (3pt);
 %   \node[\pos] at (v\i) {\i};
%}
%
%\end{tikzpicture}
%\caption{ $G$.}\label{KelmansTfig}
%\end{figure}

%\item If $\mathbb{S}_{uv}(\h)$ is not isomorphic to $\mathcal{H}$ and there exist edges $e\in E_\h(u)$ and $f\in E_\h(v)$ such that $e\cap f= \emptyset$ and $\mathbb{S}_{uv}(f)\notin E(\h)$, then $\mathbb{S}_{uv}(\h)-e$ is a proper subgraph of $\h-e$, and we get the following result by Theorem~\ref{KelmTranshypergraph}.
\end{enumerate}
\end{remark}

%Theorem~\ref{KelmTranshypergraph}
%By choice two special edge , we get the result

If $\mathbb{S}_{uv}(\h)$ is not isomorphic to $\mathcal{H}$ and there exist edges $e\in E_\h(u)$ and $f\in E_\h(v)$ such that $e\cap f= \emptyset$ and $\mathbb{S}_{uv}(f)\notin E_\h(u)$, then $\mathbb{S}_{uv}(\h)-e$ is a proper subgraph of $\h-e$.
Combining this fact and Theorem~\ref{KelmTranshypergraph},
we get the following result.
The advantage of this corollary is that it only
requires finding two special edges.

%The advantage of this corollary is that it
%we only need to find two special edges.
%other than check proper subgraph.

\begin{corollary}\label{KelmTransHGCor}
Let $\mathcal{H}$ be a connected $k$-graph with two vertices $u$ and $v$.
If $\mathbb{S}_{uv}(\h)$ is not isomorphic to $\mathcal{H}$
and there exist edges $e\in E_\h(u)$ and $f\in E_\h(v)$
%
%$e$ containing $u$ and $f$ containing $v$
%$e\in E_\h(u)\setminus E_\h(v)$ and $f\in E_\h(v)\setminus E_\h(u)$
such that $e\cap f= \emptyset$ and $\mathbb{S}_{uv}(f)\notin E_\h(u)$, then
$
\lambda(\mathcal{H})< \lambda(\mathbb{S}_{uv}(\h)).
$
\end{corollary}

%\subsection{The structure of the paper}

The rest of this paper is organized as follows.
In the next section, we give some basic properties on spectral radii and largest matching roots of
 $k$-graphs.
We prove Theorem~\ref{KelmTranshypergraph} in Section~\ref{sectionKT} and Theorem~\ref{MainresutCactus} in Section~\ref{sectioncatus}.
We conclude this paper in Section~\ref{Concluding}. % with some problems.

%\newpage

\section{Preliminaries}

In this section, we introduce some basic notation and lemmas on the spectral radii and largest matching roots of $k$-graphs.

\subsection{Basic notation}

Let $\h$ be a $k$-graph with a vertex $v$.
Denote by $E_\h(v)$ the set of all edges of $\h$ containing $v$.
The {\it degree} of $v$ is defined as $|E_\h(v)|$, denoted by $d_\h(v)$.
For a path or a cycle $\mathcal{X} = (v_0, e_1, v_1, \ldots, e_\ell, v_\ell)$, $\ell$ is called the {\it length} of $\mathcal{X}$, and we view $\mathcal{X}$ as a $k$-graph with vertex set $\cup_{i=1}^\ell e_i$ and edge set $\cup_{i=1}^\ell \{e_i\}$.
A $k$-graph $\h$ is  {\it connected} if  any pair of  vertices of $\h$ are connected by a path,
and is  a {\it $k$-uniform hypertree} (or simply {\it $k$-tree}) if $\h$ is both connected and acyclic.

Two $k$-graphs $\h_1$ and $\h_2$ are called {\it isomorphic}, denoted by $\h_1\simeq \h_2$,
if there exists a bijection $\theta:V(\h_1)\to  V(\h_2)$ such that $ \{v_1, \ldots, v_k\}\in E(\h_1)$ if and only if  $ \{\theta(v_1),  \ldots, \theta(v_k)\}\in E(\h_2)$.
We call such a bijection $\theta$ an {\it isomorphism} from $\h_1$ to $\h_2$.
The $k$-graph $\h_1$ is called a {\it subgraph} of $\h_2$
if $V(\h_1)\subseteq V(\h_2)$ and $E(\h_1)\subseteq E(\h_2)$.
If $\h_1$ is a subgraph of $\h_2$ and $\h_1\neq \h_2$,
then $\h_1$ is called a {\it proper subgraph} of $\h_2$.
If $\h_1$ is a subgraph of $\h_2$ and $V(\h_1)= V(\h_2)$,
then $\h_1$ is called a {\it spanning subgraph} of $\h_2$.
For an edge  $e\in E(\h)$, we write $\h\setminus e$ for the $k$-graph $(V(\h),E(\h)\setminus \{e\})$.
%We denote by $\h-e$ the $k$-graph obtained from $\h$ by deleting all vertices in $e$ and their incident edges, that is, $V(\h-e)=V(\h)\setminus e$ and $E(\h-e)=\{e\in E(\h):e \subseteq V(\h)\setminus e\}$.

\begin{lemma}[\cite{Suejc}]\label{basiclemma}
Let $\g$ and $\h$ be two $k$-graphs. Then the following statements hold.
\begin{enumerate}[label={\rm (\alph*)}]% (\roman*)[label=\arabic*)]
		
\item For the disjoint union $\g\oplus\h$ of $\g$ and $\h$,	
    $$\mu(\g\oplus \h, x)=\mu(\g, x)\mu(\h, x).$$
		
\item For every edge  $e\in E(\h)$,
	  $$\mu(\h, x)=\mu(\h\setminus e, x)-\mu(\h-e, x).$$
		
\end{enumerate}
\end{lemma}

\subsection{Spectral radius of $k$-graphs}

A real {\it tensor} % $ (also called \emph{hypermatrix})
$\A=(a_{i_{1}  \ldots i_{k}})$ of order $k$ and dimension $n$ is a multi-dimensional array with entries $a_{i_{1}\cdots i_{k}}\in \mathbb{R}$
for $i_{j}\in [n]:=\{1,\ldots,n\}$ and $j\in [k]$.
Let $\A=(a_{i_{1} \cdots i_{k}})$ be a tensor of order $k$ and dimension $n$.
For a vector $\x=(\x_1,\ldots,\x_n)^\top\in \mathbb{C}^{n}$, we write $\x^{[k]}=(\x_1^{k},\ldots,\x_n^{k})^\top$ and denote by $\A \x^{k-1}$ the vector in $\mathbb{C}^n$ whose $i$-th entry is defined as
\begin{align*}
(\A \x^{k-1})_i & =\sum_{i_{2},\ldots,i_{k}\in [n]}a_{ii_{2}\cdots i_{k}}\x_{i_{2}}\cdots \x_{i_k}.
\end{align*}
In 2005, Lim \cite{Lim} and Qi \cite{Qi} independently defined  the eigenvalues of tensors as follows.
For a complex number $\lambda$, if the system
$$\A \x^{k-1}=\lambda \x^{[k-1]},$$
has a solution $\x \in \mathbb{C}^{n}\setminus \{0\}$,
then $\lambda $ is said to be an \emph{eigenvalue} of $\A$ and $\x$ is called an \emph{eigenvector} of $\A$ associated with $\lambda$.

Let $\h$ be a $k$-graph with $V(\h)=\{v_1,\ldots,v_n\}$.
The {\it adjacency tensor} \cite{CD} of $\h$ is defined as $\mathcal{A}(\h)=(a_{i_{1}\cdots i_{k}})$,
a tensor of order $k$ and dimension $n$,
where
\[a_{i_{1}i_{2}\cdots i_{k}}=\left\{
 \begin{array}{ll}
\frac{1}{(k-1)!}, &  \mbox{if~} \{v_{i_{1}},\ldots,v_{i_{k}}\} \in E(\h);\\
  0, & \mbox{otherwise}.
  \end{array}\right.
\]
In this paper, the eigenvalues of a $k$-graph $\h$ always refer to those of its adjacency tensor.
The {\it spectral radius of $\h$} is defined as
\[
\rho(\h)=\max\{|\lambda|: \lambda \mbox{ is an eigenvalue of } \A(\h) \}.
\]
According to the Perron--Frobenius theorem of weakly irreducible
nonnegative tensors~\cite{Friedland,YY,YY3},
if $\h$ is connected, then there exists a unique positive eigenvector up to scaling corresponding to $\rho(\h)$, called the {\it Perron vector} of $\h$.

Li, Shao and Qi \cite{LiShaoQi} introduced the following edge-moving operation on $k$-graphs.
Let $\h$ be a $k$-graph with $u \in V(\h)$ and $e_1, \ldots, e_r \in E(\h)$ such that $u \notin e_i$ for all $i\in [r]$.
Suppose that $v_i \in e_i$
%($v_1, \ldots, v_r$ are not necessarily distinct)
and write $e'_i = (e_i \setminus \{v_i\}) \cup \{u\}$ for all $i\in [r]$.
Let $\h'$ be the $k$-graph with $V(\h')=V(\h)$ and $E(\h') = (E \setminus \{e_1, \ldots, e_r\}) \cup \{e'_1, \ldots, e'_r\}$.
We say that the $k$-graph $\h'$ is obtained from $\h$ by moving edges $e_1, \ldots, e_r$ from $v_1, \ldots, v_r$ to $u$.
Clearly, the shifting operation is a special case of the
edge-moving operation.

\begin{theorem}[\cite{LiShaoQi}] \label{moveedges}
Let $\h$ be a connected $k$-graph, and let $\h'$ be the  $k$-graph obtained from $\h$ by moving edges $e_1, \ldots, e_r$ from $v_1, \ldots, v_r$ to $u$.
Assume that $\h'$ contains no multiple edges.
If $\x$ is a Perron vector of $\h$ and $\x_u \geq \max_{1\le i \le r}\x_{v_i}$, then $\rho(\h') > \rho(\h)$.
\end{theorem}

\subsection{The $k$-walk-tree}

To refute a conjecture by Kahn and Kim~\cite{KahnKim} regarding the random matchings of $k$-graphs,
Lee~\cite{Lee} introduced the concept of $k$-walk-tree which is a hypergraph analog of the path tree \cite{Godsil2}.

\begin{definition}\label{defwalktree}
Let $\h$ be a $k$-graph with a vertex $u\in V(\h)$, and let $\prec$ be a linear ordering on $V(\h).$
\begin{itemize}
\item Let $\mathcal{P}=(v_0,e_1,v_1, \ldots, e_\ell, v_\ell)$ be a path  in $\h$, where \(\ell\geq 1\).
     For each $i\in[\ell]$, let \(e_{i}=\{v_{i-1},u_{(i,1)},\ldots,u_{(i,k-2)},v_{i}\}\) and \(C_{i}=\{u_{(i,j)}:u_{(i,j)}\prec v_{i},j\in[k-2]\}\cup\{v_{i-1}\}\).
    The path $\mathcal{P}$ is called a {\em conflict-free walk} that starts from \(v_{0}\) and ends at \(v_{\ell}\) if for each \(2\leq i\leq\ell,\) the edge \(e_{i}\) is disjoint with the set \(\bigcup_{j=1}^{i-1}C_{j}\).
    In particular, a one-vertex path $(u)$ is considered as a conflict-free walk that starts from \(u\) and ends at \(u\).

\item The {\it $k$-walk-tree} $\T(\h,\prec,u)$ of $\h$ rooted at $u\in V(\h)$ with respect to $\prec$ is the $k$-graph whose vertex set is the set of conflict-free walks that start from \(u\).
    We join \(k\) conflict-free walks \(\mathcal{P}_{0},\mathcal{P}_{1},\ldots,\mathcal{P}_{k-1}\) as an edge of $\T(\h,\prec,u)$ if there is an edge \(e=\{u_{0},\ldots,u_{k-1}\}\in E(\h)\) such that \(\mathcal{P}_{0}\) is a conflict-free walk that starts from \(u\) and ends at \(u_{0}\), and for every \(i\in[k-1]\) the conflict-free walk \(\mathcal{P}_{i}\) ends at \(u_{i}\) and is obtained from \(\mathcal{P}_{0}\) by adding the edge \(e.\)
    In particular, we simply use $u$ to denote the vertex $(u)$ in $\T(\h,\prec,u)$.

\end{itemize}
\end{definition}

As observed in~\cite{Lee}, by definition, the {\it $k$-walk-tree} $\T(\h,\prec,u)$ is a $k$-tree, and it is isomorphic to $\h$ if $\h$ is a $k$-tree.

The following result implies that the largest matching root of a $k$-graph is equal to the spectral radius of its $k$-walk-tree.

\begin{theorem}[\cite{WanWF}]\label{largestrootdeflemma}
Let $\h$ be a connected $k$-graph with a linear ordering $\prec$ on $V(\h).$
Then $\lambda(\h)$ is a simple root of $\mu(\h,x)$ and equals $\rho(\T(\h,\prec,u))$.
\end{theorem}

Let $\h,\h_1,\ldots,\h_s$ be pairwise vertex-disjoint $k$-graphs with $u_i\in V(\h)$ and $v_i\in V(\h_i)$ for all $i=1,\ldots,s$.
We denote by
$\h(u_1, \ldots,u_s)\ast(\h_1(v_1), \ldots,\h_s(v_s))$
the $k$-graph obtained from $\h$ by identifying $u_i$ with $v_i$ for each $i=1,\ldots,s$.
We briefly write $\h(u_1)\ast\h_1(v_1)$ when $s=1$.
Combining Theorems \ref{moveedges} and \ref{largestrootdeflemma}, we obtain the following result,
which suggests that
the largest matching root of a $k$-graph increases under a special edge-moving operation.

\begin{lemma}\label{treejointgraphslemma}
Let $\T,\h_1,\ldots,\h_s$ be pairwise vertex-disjoint connected $k$-graphs, where $s\geq 2$.
Let $u_1, \ldots,u_s$ be distinct vertices of $\T$ and $v_i\in V(\h_i)$ for all $i\in[s]$.
If $\T$ is a $k$-tree and $|E(\h_i)|\geq 1$ for all $i\in[s]$, then there exists $u\in\{u_1, \ldots, u_s\}$ such that
$$\lambda(\T(u_1, \ldots,u_s)\ast(\h_1(v_1), \ldots,\h_s(v_s)))
<
\lambda(\T(u, \ldots,u)\ast(\h_1(v_1), \ldots,\h_s(v_s))).$$
\end{lemma}
\begin{proof}
By recursiveness, it suffices to prove the case of $s=2$.
Write $\h:=\T(u_1, u_2)\ast(\h_1(v_1), \h_2(v_2))$, and let $v\in V(\T)$.
We fix a linear order $\prec$ on $\h$.
As $\T$ is a $k$-tree, one may check that
$$\T(\h,\prec,v)\simeq \T(u_1, u_2)
\ast
(\T(\h_1,\prec,v_1)(v_1),\T(\h_2,\prec,v_2)(v_2))
=:\T_1,$$
Given a Perron vector $\x$ of $\T_1$, without loss of generality,
we may assume that $\x_{u_1}\geq \x_{u_2}$.
We choose $u=u_1$ which is the required vertex.
Write $\widetilde{\h}:=\T(u_1,u_1)\ast(\h_1(v_1),\h_2(v_2))$.
Then one may check that
$$\T(\widetilde{\h},\prec,v) \simeq \T(u_1, u_1)\ast (\T(\h_1,\prec,v_1)(v_1),\T(\h_2,\prec,v_2)(v_2))
=:\T_2.$$
Moreover, one may obtain $\T_2$ from $\T_1$ by moving every edge of $\T(\h_2,\prec,v_2)$ adjacent to $v_2$ from $v_2$ to $v_1$ in $\T_1$.
By Theorems~\ref{moveedges} and \ref{largestrootdeflemma}, we obtain
$$\lambda(\h)=\rho(\T_1)
< \rho(\T_2)=\lambda(\widetilde{\h}),$$
which completes the proof.
\end{proof}

\subsection{$k$-cacti}

Recall the definition of the $k$-cactus in Subsection~\ref{Mainsubsection}.
To facilitate understanding, we give the following observation.
\begin{observation}\label{observationcactus}%[\cite{Suejc}]\label{basiclemma}
Let $k\geq 3$ and let $\h$ be a $k$-cactus.
Then the following configurations are forbidden in $\h$.
\begin{enumerate} [label={\rm (\alph*)}]% (\roman*)[label=\arabic*)]
		
\item Two edges share at least three common vertices.
		
\item Three edges share at least two common vertices.

\item Two cycles share at least one common edge.

\end{enumerate}
\end{observation}

The following lemma gives more structural information for $k$-cacti.

\begin{lemma}\label{lemmacactus}%[\cite{Suejc}]\label{basiclemma}
Let $\h$ be a $k$-cactus with a cycle $\mathcal{C}$.
Then the following assertions hold.
\begin{enumerate} [label={\rm (\alph*)}]%

\item If the length of $\mathcal{C}$ is at least three, then it is linear.

 \item We can write  $\h$ as
\begin{equation}\label{decompcat}
\h=\mathcal{C}(v_1, \ldots,v_s)\ast(\h_1(v_1), \ldots,\h_s(v_s)),
\end{equation}
 where $v_1, \ldots,v_s\in V(\mathcal{C})$ and $\h_1,\ldots,\h_s$ are pairwise vertex-disjoint connected $k$-graphs.

 \end{enumerate}
\end{lemma}
\begin{proof}
If $\mathcal{C}$ is not linear, then it contains two edges $e_1$ and $e_2$ such that $|e_1\cap e_2|\geq 2$.
Let $\{u,v\}\subseteq e_1\cap e_2$.
Since  the length of $\mathcal{C}$ is at least three,
$\mathcal{C}$ and the cycle $(u,e_1,v,e_2,u)$ are distinct and they share two common edges $e_1$ and $e_2$.
This contradicts Observation~\ref{observationcactus}(c), and (a) follows.

We next prove (b).
If $\mathcal{X} = (v_0, e_1, v_1, \ldots, e_\ell, v_\ell)$ is a path, then we call it a $v_0v_\ell$-path.
Choose an arbitrary vertex $v\in V(\mathcal{C})$
and assume that $\h$ cannot be written as $\h_1(v)\ast\h_2(v)$, where $\h_1$ contains $\mathcal{C}$ and $E_{\h_1}(v)=E_{\mathcal{C}}(v)$.
By the arbitrariness of $v$ and Observation~\ref{observationcactus}(c), to get a contradiction, it suffices to show that there exist an edge $e\notin E(\mathcal{C})$ and a cycle $\widehat{\mathcal{C}}$ with $e\in \widehat{\mathcal{C}}$
such that $\widehat{\mathcal{C}}$ and $\mathcal{C}$ share at least one edge.
We show this by considering the following two cases.
(1) If there exist an edge $e\in E_\h(v)\setminus E(\mathcal{C})$ and a vertex $u\in V(\mathcal{C})\setminus \{v\}$ such that $u\in e$,
then $e$ and a $uv$-path $\mathcal{P}$ with $E(\mathcal{P})\subseteq E(\mathcal{C})$ form a required cycle $\widehat{\mathcal{C}}$.
(2) If for every edge $f\in E_\h(v)\setminus E(\mathcal{C})$, we have $u\notin V(\mathcal{C})$ for each $u\in f\setminus \{v\}$,
then one may check that there exist an edge $e\in E_\h(v)\setminus E(\mathcal{C})$,
a vertex $w\in e\setminus \{v\}$, a vertex $z\in V(\mathcal{C})\setminus \{v\}$,
and a $wz$-path $\mathcal{P}_{wz}$ such that $v\notin V(\mathcal{P}_{wz})$ and $e\notin E(\mathcal{P}_{wz})$.
In this case, the edge $e$, the path $\mathcal{P}_{wz}$, and a $zv$-path $\mathcal{P}_{zv}$ with $E(\mathcal{P}_{zv})\subseteq E(\mathcal{C})$ form a required cycle $\widehat{\mathcal{C}}$.
\end{proof}

%\newpage

\section{Proof of Theorem~\ref{KelmTranshypergraph}}
\label{sectionKT}

In~\cite{LovaszPe}, Lov\'asz and  Pelik\'an introduced the ordering on forests based on the matching polynomial.
In recent years, the ordering was generalized to graphs by Csikv\'ari~\cite{Csikvari} and Chen, Li, and Lian~\cite{Chen}, and to $k$-forests by Su, Kang, Li, and Shan~\cite{Suejc}.
Inspired by their works,
we apply the same principle to define an ordering for $k$-graphs via the matching polynomials, which will be used to prove Theorem~\ref{KelmTranshypergraph}.

\begin{definition}
Let $\g$ and $\h$ be two $k$-graphs.
We write $\g\prec\h$ if $\mu(\g, x)> \mu(\h, x)$ for every $x\geq \lambda(\h)$ and write $\g\preccurlyeq \h$ if $\mu(\g, x)\geq \mu(\h, x)$ for every $x\geq \lambda(\h)$.
\end{definition}

Note that the matching polynomial is monic,
we can easily verify the following lemma.

\begin{lemma}\label{succrelation}
Let $\g$ and $\h$ be two $k$-graphs.
If $\g\preccurlyeq\h$, then
$\lambda(\g)\leq\lambda(\h)$.
If $\g\prec\h$, then
$\lambda(\g)<\lambda(\h)$.

\end{lemma}

To obtain more properties of the ordering, we need the following lemma.

\begin{lemma}[Lemma 4.5~\cite{WanWF}]\label{subgraphlagroot}
%Let $\h$ be a connected $k$-graph.
If $\G$ is a subgraph of a connected $k$-graph $\h$, then
$\lambda(\G)\leq \lambda(\h)$,
with equality if and only if $\G=\h$.
\end{lemma}

\begin{lemma}\label{spanningsucclemma}
If $\g$ is a proper spanning subgraph of a connected $k$-graph $\h$, then $\g\prec\h$.
\end{lemma}
\begin{proof}
By the hypothesis, there exist distinct edges $\{e_1,\dots,e_t\}$ such that $\g=\h\setminus \{e_1,\dots,e_t\}$,
where $t\geq 1$ and $\h\setminus \{e_1,\dots,e_t\}$ denotes the $k$-graph $(V(\h),E(\h)\setminus \{e_1,\dots,e_t\})$.
Write $e_0:= \emptyset$.
Repeatedly applying Lemma~\ref{basiclemma}(b), % $t$ times,
we obtain
\begin{align*}
\mu(\h, x)
&=\mu(\h\setminus e_1, x)-\mu(\h-e_1, x)\\
&=\mu(\h\setminus \{e_1,e_2\}, x)-\mu(\h\setminus e_1-e_2, x)    -\mu(\h- e_1, x)\\
&=\cdots\\
&=\mu(\h\setminus \{e_1,\dots,e_t\}, x)-\sum_{i=1}^t\mu(\h \setminus\{e_0,e_1,\dots,e_{i-1}\}-e_i, x),
\end{align*}
which implies that
\begin{equation}\label{properconnectedeq}
\mu(\g,x)-\mu(\h,x)
= \sum_{i=1}^t\mu(\h \setminus\{e_0,e_1,\dots,e_{i-1}\}-e_i,x).
\end{equation}
For every $i\in[t]$, since $\mu(\h \setminus\{e_0,e_1,\dots,e_{i-1}\}-e_i, x)$ is monic,
we get $\mu( \h \setminus\{e_0,e_1,\dots,e_{i-1}\}-e_i , x)>0$ whenever $x>\lambda(\h\setminus\{e_0,e_1,\dots,e_{i-1}\}-e_i )$.
Moreover, $\h \setminus\{e_0,e_1,\dots,e_{i-1}\}-e_i$ is a proper subgraph of $\h$, so $\lambda(\h)>\lambda(\h \setminus\{e_0,e_1,\dots,e_{i-1}\}-e_i)$ by Lemma~\ref{subgraphlagroot}.
Combining this and \eqref{properconnectedeq}, we conclude that $\mu(\g, x)-\mu(\h, x)>0$
for every $x\geq \lambda(\h)$, as desired.
\end{proof}

\begin{corollary}\label{spanningsucclemma22}
If $\g$ is a proper spanning subgraph of a $k$-graph $\h$,
then  $\mu(\g, x)> \mu(\h, x)$ for every $x> \lambda(\h)$.
%and hence that $\g\preccurlyeq \h$.
Moreover, if $\widehat{\g}$ is a spanning subgraph of $\h$, then $\widehat{\g}\preccurlyeq \h$.
\end{corollary}
\begin{proof}
Assume that $\g$ is a proper spanning subgraph of a $k$-graph $\h$.
If $\h$ is connected, then the assertion follows from Lemma~\ref{spanningsucclemma}.
Suppose that $\h$ has $t\geq 2$ components $\h_1,\ldots,\h_t$.
Then $\g$ can be written as the disjoint union of
$\g_1,\ldots,\g_t$, where $\g_i$ is a spanning subgraph of $\h_i$ for every $1\leq i\leq t$.
Without loss of generality, one may assume that there exists an integer $1\leq j\leq t$ such that $\g_i$ is a proper spanning subgraph of $\h_i$ for every $1\leq i\leq j$ and $\g_i=\h _i$ for every $j< i\leq t$.
For every $1\leq i\leq j$,
we have $\g_i \prec \h_i$ by Lemma~\ref{spanningsucclemma},
which implies that $\mu(\g_i, x)> \mu(\h_i, x)>0$
when $x> \lambda(\h_i)$.
Also, for every $j< i\leq t$, we have $\mu(\g_i, x)=\mu(\h_i, x)>0$ when $x> \lambda(\h_i)$.
By Lemma~\ref{basiclemma}(a), we have $\lambda(\h)=\max_{1\leq i\leq t} \lambda(\h_i)$.
Hence, for every $x> \lambda(\h)=\max_{1\leq i\leq t} \lambda(\h_i)$, we have
$$\mu(\g, x)=\prod_{1\leq i\leq t}\mu(\g_i, x)> \prod_{1\leq i\leq t}\mu(\h_i, x) =\mu(\h, x),$$
as desired.
Using this, we deduce that $\mu(\g, x)\geq \mu(\h, x)$ for every $x\geq \lambda(\h)$, that is ${\g}\preccurlyeq \h$.
This implies that $\widehat{\g}\preccurlyeq \h$ if $\widehat{\g}$ is a spanning subgraph of $\h$.
\end{proof}

\begin{lemma}\label{Kelmanstprelma}
Let $\mathcal{H}$ be a $k$-graph with two vertices $u$ and $v$.
Then for every edge $e\in E_\h(u)\setminus E_\h(v)$, we have
$$\lambda(\mathbb{S}_{uv}(\h)) \geq \max\left \{\lambda(\mathbb{S}_{uv}(\h)\setminus e),\lambda(\h- e)\right \}.$$
Moreover, the inequality is strict when $\mathcal{H}$ is connected.
\end{lemma}
\begin{proof}
Note that $e=\mathbb{S}_{uv}(e)\in E(\mathbb{S}_{uv}(\h))$,
so $\mathbb{S}_{uv}(\h)\setminus e$ is a proper subgraph of $\mathbb{S}_{uv}(\h)$.
By Lemma~\ref{subgraphlagroot}, we have $\lambda(\mathbb{S}_{uv}(\h))
\geq \lambda(\mathbb{S}_{uv}(\h)\setminus e)$.
To prove $\lambda(\mathbb{S}_{uv}(\h))\geq \lambda(\h- e)$, we
define the map
$\phi: E(\h- e) \to E(\mathbb{S}_{uv}(\h))$ by
\[\phi(g)
=\left\{
 \begin{array}{ll}
g, &  \mbox{if~} g\in E(\h- e) {\rm~and~} v\notin g;\\
(g\setminus \{v\})\cup\{u\}, & \mbox{if~} g\in E(\h- e) {\rm~and~} v\in g.
  \end{array}\right.
\]
Observe that $\phi$ is injective, and we next show that $\phi$ is not surjective.
If $(f\setminus \{v\})\cup\{u\}= e$ for an edge $f\in E(\h- e)$ with $v\in f$, then $f\notin \phi(E(\h- e))$ and $f\in E(\mathbb{S}_{uv}(\h))$.
If $(g\setminus \{v\})\cup\{u\}\neq e$ for every edge $g\in E(\h- e)$ with $v\in g$, then $e \notin \phi(E(\h- e))$.
This implies that $\phi$ is not surjective.
So, the $k$-graph $\mathcal{G}=( (V(\h- e)\setminus\{v\}) \cup \{u\}, \phi(E(\h- e))$ is a proper subgraph of $\mathbb{S}_{uv}(\h)$.
Clearly, the map $\varphi: V(\h- e)\to V(\mathcal{G})$ defined by
\[\varphi(w)
=\left\{
 \begin{array}{ll}
u, &  \mbox{if~} w=v;\\
w, & \mbox{if~} w\neq v.
  \end{array}\right.
\]
is an isomorphism between $\h-e$ and $\mathcal{G}$.
Thus, $\h-e$ is isomorphic to a proper subgraph of $\mathbb{S}_{uv}(\h)$.
We deduce that
$\lambda(\mathcal{G})\leq \lambda(\mathbb{S}_{uv}(\h))$
by Lemma~\ref{subgraphlagroot}.

We further assume that $\h$ is connected.
Then either $\mathbb{S}_{uv}(\h)$ is connected or $\mathbb{S}_{uv}(\h)$ is comprised of $v$ and a component, say $\g$.
If $\mathbb{S}_{uv}(\h)$ is connected, then the above two inequalities are strict by Lemma~\ref{subgraphlagroot}.
Otherwise, one may check that
$$\lambda(\mathbb{S}_{uv}(\h))=\lambda(\g)
> \max\left \{\lambda(\mathbb{S}_{uv}(\h)\setminus e),\lambda(\h- e)\right \}$$
by Lemma~\ref{subgraphlagroot}.
\end{proof}

%\newpage

We are now ready to show the main result of this section.
As a consequence, Theorem~\ref{KelmTranshypergraph} follows from  Lemma~\ref{succrelation}.

\begin{theorem}\label{Kelmansttheorem11}
Let $\mathcal{H}$ be a $k$-graph with two vertices $u$ and $v$.
Then
$$\h\preccurlyeq \mathbb{S}_{uv}(\mathcal{H}).$$
Moreover, if $\mathcal{H}$ is connected and is not isomorphic to $\mathbb{S}_{uv}(\mathcal{H})$, and there exists an edge $e\in E_{\mathcal{H}}(u)\setminus E_{\mathcal{H}}(v)$ such that $\mathbb{S}_{uv}(\mathcal{H})-e$ is a proper subgraph of $\mathcal{H}-e$, then $\mathcal{H} \prec \mathbb{S}_{uv}(\mathcal{H})$.
\end{theorem}

\begin{proof}
%Let $\mathcal{H}$ be a $k$-graph with two vertices $u$ and $v$.
We prove $\h\preccurlyeq \mathbb{S}_{uv}(\mathcal{H})$ by induction on the size of $|E(\h)|$.
If $|E(\h)|=1$, then the statement is valid as $\mathbb{S}_{uv}(\h)\simeq\mathcal{H}$.
Assume that $|E(\h)|\geq 2$ and the statement holds for all $k$-graphs with size less than $|E(\h)|$.
If $\mathbb{S}_{uv}(\h)\simeq\mathcal{H}$, then $\h\preccurlyeq \mathbb{S}_{uv}(\mathcal{H})$ is trivial.
If $\mathbb{S}_{uv}(\h)$ is not isomorphic to $\mathcal{H}$,
then there exist an edge $e\in E_\h(u)\setminus E_\h(v)$.
%and an edge $f\in E_\h(v)\setminus E_\h(u)$ satisfying $\mathbb{S}_{uv}(f)\notin E(\h)$.
By Lemma~\ref{basiclemma}(b), we have
\begin{equation}\label{detedheKteq}
\mu(\h, x)-\mu(\mathbb{S}_{uv}(\h), x)
= \mu(\h\setminus e, x)-\mu(\mathbb{S}_{uv}(\h)\setminus e, x)
+\mu(\mathbb{S}_{uv}(\h)-e, x)-\mu(\h-e, x).
\end{equation}
Recall that $e\in E_\h(u)\setminus E_\h(v)$,
so $\mathbb{S}_{uv}(e)=e$, and hence that $\mathbb{S}_{uv}(\h\setminus e)=\mathbb{S}_{uv}(\h)\setminus e$.
As $|E(\h\setminus e)|<|E(\h)|$, the induction hypothesis gives
${\h}\setminus e\preccurlyeq \mathbb{S}_{uv}(\h\setminus e)$, and we deduce that
\begin{equation}\label{detedheKteq22331}
\mu(\h\setminus e, x)-\mu(\mathbb{S}_{uv}(\h)\setminus e, x)\geq 0 {\rm ~whenever~}x\geq \lambda(\mathbb{S}_{uv}(\h)\setminus e).
\end{equation}
Also, one may check that $\mathbb{S}_{uv}(\h)-e$ is a spanning subgraph of $\h-e$.
By Corollary \ref{spanningsucclemma22},
we get $\mathbb{S}_{uv}(\h)-e \preccurlyeq\h-e$, that is,
\begin{equation}\label{detedheKteq22332}
\mu(\mathbb{S}_{uv}(\h)- e, x)- \mu(\h-e, x)\geq 0
{\rm ~whenever~} x \geq \lambda(\h- e).
\end{equation}
Applying~\eqref{detedheKteq22331} and~\eqref{detedheKteq22332}  to~\eqref{detedheKteq},
we deduce that
$$%\begin{equation}\label{detedheKteq22}
\mu(\h, x)-\mu(\mathbb{S}_{uv}(\h), x)\geq 0
{\rm ~whenever~} x\geq \max\left\{\lambda(\mathbb{S}_{uv}(\h)\setminus e), \lambda({\h}- e)\right\}.
$$%\end{equation}
Combining this fact and Lemma~\ref{Kelmanstprelma}, we get
$\h\preccurlyeq \mathbb{S}_{uv}(\mathcal{H})$, which completes the induction.

Moreover, if $\mathcal{H}$ is connected and is not isomorphic to $\mathbb{S}_{uv}(\mathcal{H})$,
and $\mathbb{S}_{uv}(\h)-e$ is a proper spanning subgraph of $\h-e$,
then by Corollary \ref{spanningsucclemma22} we get that
\begin{equation}\label{detedheKteq22333}
\mu(\mathbb{S}_{uv}(\h)- e, x)- \mu(\h-e, x)> 0
{\rm ~whenever~} x> \lambda(\h- e).
\end{equation}
Applying~\eqref{detedheKteq22331} and~\eqref{detedheKteq22333}  to~\eqref{detedheKteq},
we deduce that
$$%\begin{equation}\label{detedheKteq22}
\mu(\h, x)-\mu(\mathbb{S}_{uv}(\h), x)> 0
{\rm ~whenever~} x> \max\left\{\lambda(\mathbb{S}_{uv}(\h)\setminus e), \lambda({\h}- e)\right\}.
$$%\end{equation}
Combining this fact and Lemma~\ref{Kelmanstprelma}, we get
$\h\prec \mathbb{S}_{uv}(\mathcal{H}).$
This completes the proof.
\end{proof}

%\newpage

\section{Proof of Theorem~\ref{MainresutCactus}}
\label{sectioncatus}

This section is devoted to proving Theorem~\ref{MainresutCactus}.
We begin with the following lemma which will be used later.

\begin{lemma}\label{degreeonelemma}
Let $\g$ and $\h$ be two vertex-disjoint connected $k$-graphs with $u,v\in V(\g), w\in V(\h)$ and $|E(\h)|\geq 1$.
If $u$ and $v$ are adjacent in $\g$ such that $d_\g(u)>1$ and $d_\g(v)=1$,
then %we have %$\g(v)\ast \h(w)\prec\g(u)\ast \h(w).$
%In particular,
$$\lambda(\g(v)\ast \h(w))< \lambda(\g(u)\ast \h(w)).$$
\end{lemma}
\begin{proof}
Note that
$\mathbb{S}_{uv}(\g(v)\ast \h(w))
=\g(u)\ast \h(w)$ and it is not isomorphic to $\g(v)\ast \h(w)$.
According to the hypothesis on $u$ and $v$, we can pick edges $e\in E_\g(u)\setminus E_\g(v)$ and $f\in E_\h(w)$.
Observe that $e\cap f= \emptyset$ and $\mathbb{S}_{uv}(f)\notin E_{\g(v)\ast \h(w)}(u)$,
so the result follows from Corollary~\ref{KelmTransHGCor}.
\end{proof}

A $k$-graph is called a {\it $k$-star} with $m$ edges, denoted by $\mathcal{S}_m^{(k)}$,
if its vertex set admits a partition
$\{v\}\cup V_1 \cup\cdots  \cup V_m $
with $|V_1| =\cdots =|V_m| =k-1$
and the edge set is $\{\{v\}\cup V_i : i = 1,\ldots,m\}$.
The vertex $v$ is called the {\it center} of $\mathcal{S}_m^{(k)}$.

Let $\h$ be a $k$-graph.
An edge $e\in E(\h)$ is called a {\it pendant edge} of $\h$ if it contains at least $k-1$ vertices of degree one.
If $\h$ can be written as $\G(u)\ast \T(v)$, where $\G$ is a $k$-graph and $\T$ is a $k$-tree, then we call $\T$ a {\it pendant $k$-tree} of $\h$ with respect to the {\it root} $v$.
A pendant $k$-tree is said to be  {\it maximal} if it is not a proper subgraph of any other pendant $k$-tree.

\begin{lemma}\label{pendanttreelemma}
Let $\h$ and $\T$ be two vertex-disjoint connected $k$-graphs, and let $u\in V(\h)$ and $v\in V(\T)$.
If $\T$ is a $k$-tree with $m$ edges, then
$$\lambda(\h(u)\ast \T(v))
\leq
\lambda(\h(u)\ast \mathcal{S}^{(k)}_m(v)),$$
where $v$ is the center of $\mathcal{S}^{(k)}_m$,
and the equality holds if and only if
$\h(u)\ast \T(v) \simeq
\h(u)\ast \mathcal{S}^{(k)}_m(v)$.
\end{lemma}
\begin{proof}
We begin with the case of $|E(\h)|\geq 1$.
Write $$\mathfrak{H}_m=\{\h(u)\ast \mathcal{F}(v): \mathcal{F} \text{ is a $k$-tree with $m$ edges and}~v\in V(F)\}.$$
Let $\h(u)\ast \T(v)$ be a maximizing $k$-graph in $\mathfrak{H}_m$.
If $\T$ is not a $k$-star with center $v$, then there exist two edges $e,f\in E(\T)$ and a vertex $z$ such that
$v\in e$, $v\notin f$ and $e\cap f =\{z\}$.
As $\T$ is a $k$-tree, $\h(u)\ast \T(v)$ can be uniquely written as $\widehat{\h}(z)\ast \widehat{\T}(z)$,
where $\widehat{\T}$ is a subtree of $\T$ and $\widehat{\h}$ is a $k$-graph containing $\h$ such that $d_{\widehat{\h}}(z)=1$.
By Lemma~\ref{degreeonelemma}, we get
$$\lambda(\widehat{\h}(z)\ast \widehat{\T}(z))<\lambda(\widehat{\h}(v)\ast \widehat{\T}(z)).$$
Clearly, we have $\widehat{\h}(v)\ast \widehat{\T}(z)\in \mathfrak{H}_m$, which contradicts the maximality of $\h(u)\ast \T(v)$ in $\mathfrak{H}_m$.
%Clearly, we have $\mathbb{S}_{vz}(\h(u)\ast \T(v))\in \mathfrak{H}_m$, and $f\cap g=\emptyset$ for every $g\in E_{\h}(v)$.
% Theorem~\ref{KelmTranshypergraph} implies that
%$\lambda(\h(u)\ast \T(v))<\lambda(\mathbb{S}_{vz}(\h(u)\ast \T(v)))$, which contradicts the maximality of $\h(u)\ast \T(v)$ in $\mathfrak{H}_m$.
%++++++++++
For the case of $|E(\h)|=0$, we have $\h(u)\ast \T(v)=\T$.
If $\T$ is not a $k$-star, then we can write $\T=\T_1(w)\ast \T_2(w)$, where $|E(\T_1)|\geq 1$ and $|E(\T_2)|\geq 1$.
Applying the obtained result, we get that $\T_1$ and $\T_2$ are both $k$-stars with center $w$, as desired.
\end{proof}

As an application of Theorem~\ref{largestrootdeflemma} and Lemma~\ref{pendanttreelemma}, we have the following important result in spectral hypergraph theory.

\begin{corollary}[\cite{LiShaoQi}]\label{extremaltreecorollary}
If $\T$ is a $k$-tree with $m$ edges, then
$$\rho(\T)\leq \rho(\mathcal{S}^{(k)}_m),$$ with equality if and only if
$\T\simeq\mathcal{S}^{(k)}_m$.
\end{corollary}

The following lemma gives some structural
information on the maximizing $k$-graphs in $\mathfrak{Ca}_k(m,t)$ and $\mathfrak{LCa}_k(m,t)$ for $m\geq 4$.

\begin{lemma}\label{structurallma}
Let $k\geq 3,$ $m\geq 4,$ $t\geq 0$ be integers, and let $\h$ be a maximizing $k$-graph in $\mathfrak{Ca}_k(m,t)$ {\rm(}resp. $\mathfrak{LCa}_k( m,t)${\rm)}.
Then the following assertions hold.
\begin{enumerate} [label={\rm (\alph*)}]% (\roman*)[label=\arabic*)]
		
\item Every edge of $\h$ lies in either a cycle or a pendant tree.

\item For every cycle $\mathcal{C}$ of $\h$, if
$v\in V(\mathcal{C})$ and $d_\mathcal{C}(v)=1$, then  $d_\h(v)=1$. In particular, if two cycles $\mathcal{C}_1$ and $\mathcal{C}_2$ of $\h$ have a common vertex $u$,
     then $d_{\mathcal{C}_1}(u)=d_{\mathcal{C}_2}(u)=2$.

\item Every cycle of $\h$ has length $2$ {\rm(}resp. 3{\rm)}.

\end{enumerate}
\end{lemma}
\begin{proof}
For (a) and (b), we prove the case of $\h\in \mathfrak{Ca}_k(m,t)$,
and the case of $\h\in \mathfrak{LCa}_k(m,t)$ follows similarly.
We begin with the proof of (a).
If there exists an edge $e$ that does not satisfy the assertion, then $\h$ can be written as $\T(u_1, \ldots,u_s)\ast(\h_1(v_1), \ldots,\h_s(v_s))$, where $\T$ is a $k$-tree containing the edge $e$, $s\geq 2$, and every $\h_i$ contains a cycle for $i\in[s]$.
By Lemma~\ref{treejointgraphslemma},
there exists $u\in\{u_1, \ldots, u_s\}$ such that
$$\lambda(\T(u_1, \ldots,u_s)\ast(\h_1(v_1), \ldots,\h_s(v_s)))
<
\lambda(\T(u, \ldots,u)\ast(\h_1(v_1), \ldots,\h_s(v_s))).$$
Note that $\T(u, \ldots,u)\ast(\h_1(v_1), \ldots,\h_s(v_s))$ also lies in $\mathfrak{Ca}_k(m,t)$, which contradicts the maximality of $\h$ in $\mathfrak{Ca}_k(m,t)$, as desired.

For (b), if $d_\mathcal{C}(v)=1$ but $d_\h(v)>1$, by Lemma~\ref{lemmacactus}(b), then $\h$ can be written as $\h_1(v)\ast \h_2(v)$,
where $\h_1$ contains $\mathcal{C}$ with $d_{\h_1}(v)=1$ and $\h_2$ contains at least one edge.
Let $v'\in V(\mathcal{C})$ be a vertex such that $d_\mathcal{C}(v')=2$ and $v$ and $v'$ are adjacent in $\h_1$.
Then Lemma~\ref{degreeonelemma} implies that
$\lambda(\h_1(v)\ast \h_2(v))< \lambda(\h_1(v')\ast \h_2(v))$.
Since $\h_1(v')\ast \h_2(v)$ lies in $\mathfrak{Ca}_k(m,t)$,
this contradicts the maximality of $\h$ in $\mathfrak{Ca}_k(m,t)$ and thus yields (b).

It remains to prove (c).
Let $\h$ be a maximizing $k$-graph in $\mathfrak{Ca}_k(m,t)$, and
let $\mathcal{C}=(v_0,e_1,v_1, \ldots, e_\ell, v_\ell)$ be a cycle of $\h$ with length $\ell$.
For $\ell \geq 3$, observe that $\mathbb{S}_{v_0v_1}(\h)\in \mathfrak{Ca}_k(m,t)$ and it is not isomorphic to $\h$ by Lemma~\ref{lemmacactus}(b).
If $\ell \geq 4$, then $e_2\cap e_\ell=\emptyset$ and $\mathbb{S}_{v_0v_1}(e_2)\notin E_\h(v_0)$, so Corollary~\ref{KelmTransHGCor} implies that
$\lambda(\h)< \lambda(\mathbb{S}_{v_0v_1}(\h))$, which  contradicts the maximality of $\h$ in $\mathfrak{Ca}_k(m,t)$.
Hence, every cycle of $\h$ has length at most $3$.
Note that for the case of $\h\in \mathfrak{LCa}_k(m,t)$,
this discussion still remains valid, so we get that every cycle of $\h$ has length $3$ in this case.
If $\ell =3$, note that $m\geq 4$, then Lemma~\ref{structurallma}(b) and Lemma~\ref{lemmacactus}(b) suggests that there exists an edge $f\in E(\h)\setminus \{e_1,e_2,e_3\}$ such that $|f\cap \{v_0,v_1,v_2\}|=1$.
Without loss of generality, we assume that $f\cap \{v_0,v_1,v_2\}=\{v_0\}$.
Then $f\cap e_2=\emptyset$ and $\mathbb{S}_{v_0v_1}(e_2)\notin E_\h(v_0)$, so Corollary~\ref{KelmTransHGCor} implies that
$\lambda(\h)< \lambda(\mathbb{S}_{v_0v_1}(\h))$, which  contradicts the maximality of $\h$ in $\mathfrak{Ca}_k(m,t)$.
Thus, every cycle of $\h$ has length at most $2$, as desired.
\end{proof}

Let $k\geq 3, t\geq 0$ and $m\geq 2t$ {\rm(}resp. $m\geq 3t${\rm)} be integers.
We write $\h^{(k)}_{m,t}$ {\rm(}resp. $\mathcal{L}^{(k)}_{m,t}${\rm)} for the unique $k$-graph in $\mathfrak{Ca}_k(m,t)$ {\rm(}resp. $\mathfrak{LCa}_k(m,t)${\rm)} satisfying the following conditions:
\begin{enumerate}  %[label={\rm (\alph*)}]% (\roman*)[label=\arabic*)]

\item $\h^{(k)}_{m,t}$ has a vertex $u$ and exactly $m-2t$ {\rm(}resp. $m-3t${\rm)} pendant edges containing $u$.

\item For every cycle $\mathcal{C}$ of $\h^{(k)}_{m,t}$, $\mathcal{C}$ has length $2$ {\rm(}resp. $3${\rm)} and
    $d_\mathcal{C}(u)=2$.

\end{enumerate}
The structure of $\h^{(k)}_{m,t}$ {\rm(}resp. $\mathcal{L}^{(k)}_{m,t}${\rm)} is illustrated in Figure~\ref{figure3}.
We next show that  $\h^{(k)}_{m,t}$ is the
unique maximizing $k$-graph in $\mathfrak{Ca}_k(m,t)$ for $k\geq 4$.

\begin{theorem}\label{MainresutCactus111}
Let $k\geq 3,$ $m\geq 4$ and $t\geq 0$ be integers.
The $k$-graph $\h^{(k)}_{m,t}$ %illustrated in Figure~\ref{figure3}
is the unique maximizing $k$-graph in $\mathfrak{Ca}_k(m,t)$.
\end{theorem}
\begin{proof}
Let $\h$ be a maximizing $k$-graph in $\mathfrak{Ca}_k(m,t)$.
For the case of $t=0$, the statement follows from Lemma~\ref{pendanttreelemma}.
Assume that $t\geq 1$.
To prove that $\h$ is isomorphic to $\h^{(k)}_{m,t}$,
it suffices to prove that $\h$ satisfies the following properties.

\begin{enumerate}[label={\rm (\alph*)}]%

\item[(A1)] Every edge of $\h$ lies in either a cycle or a pendant tree.

\item[(A2)] Every cycle of $\h$ has length $2$.

\item[(A3)] There exists a vertex $u\in V(\h)$ such that $d_\mathcal{C}(u)=2$ for every cycle $\mathcal{C}$ of $\h$.

\item[(A4)] Every pendant edge of $\h$ contains the vertex $u$.

\end{enumerate}

Clearly, (A1) and (A2) follow from Lemma~\ref{structurallma}(a) and (c), respectively.
For (A3), it is trivial for the case of $t= 1$, and we further
assume $t\geq 2$.
By (A1), there exist two cycles $\mathcal{C}_1$ and $\mathcal{C}_2$ such that they have a common vertex, say $u$.
We next prove that $u$ is the required vertex.
 Lemma~\ref{structurallma}(b) implies that for every cycle $\mathcal{C}$ of $\h$, either
$d_\mathcal{C}(u)=2$ or $\mathcal{C}$ does not contain $u$,
and so $d_{\mathcal{C}_1}(u)=d_{\mathcal{C}_2}(u)=2$.
If there exists a cycle of $\h$ containing no $u$,
then (A1) suggests that there are two cycles $\mathcal{C}_3$ and $\mathcal{C}_4$ such that $d_{\mathcal{C}_3}(u)=2$, $V(\mathcal{C}_4)$ does not contain $u$,
and $V(\mathcal{C}_3)$ and $V(\mathcal{C}_4)$ have a common vertex, say $z$.
By (A2), the length of $\mathcal{C}_3$ is $2$, so $u$ and  $z$ are adjacent in $\h$.
Combining this and Lemma~\ref{lemmacactus}(b), one may check that $\mathbb{S}_{uz}(\h) \in \mathfrak{Ca}_k(m,t)$ and it is not isomorphic to $\h$.
Moreover, we pick edges $e\in E_{\h}(u)\setminus E(\mathcal{C}_3)$ and $f \in E(\mathcal{C}_4)$, then
$e \cap f=\emptyset$ and $\mathbb{S}_{uz}(f)\notin E_\h(u)$, so Corollary~\ref{KelmTransHGCor} implies that
$\lambda(\h)< \lambda(\mathbb{S}_{uz}(\h))$, which  contradicts the maximality of $\h$ in $\mathfrak{Ca}_k(m,t)$.
Hence, $d_\mathcal{C}(u)=2$ for every cycle $\mathcal{C}$ of $\h$, as desired.

It remains to prove (A4).
By Lemma~\ref{structurallma}(a), if $e$ is a pendant edge of $\h$, then it is contained in a
maximal pendant $k$-tree of $\h$, say $\T$.
Lemma~\ref{pendanttreelemma} implies that $\T$ is a $k$-star
and the center $w$ of $\T$ is contained in a cycle $\mathcal{C}$ of $\h$ by Lemma~\ref{structurallma}(a).
Moreover, Lemma~\ref{structurallma}(b) implies that $d_\mathcal{C}(w)=2$.
If $w\neq u$, then $w$ and $u$ are adjacent in $\mathcal{C}$ as (A1) says that the length of $\mathcal{C}$ equals $2$.
By Lemma~\ref{lemmacactus}(b), one may check that $\mathbb{S}_{uw}(\h) \in \mathfrak{Ca}_k(m,t)$.
If $t=1$, then either $\mathbb{S}_{uw}(\h) \simeq \h$ or $\mathbb{S}_{uw}(\h)$ is not isomorphic to $\h$.
In the latter case, there exists an edge $f\in E_\h(u)\setminus E(\mathcal{C})$.
Then $e \cap f=\emptyset$ and $\mathbb{S}_{uw}(e)\notin E_\h(u)$, so Corollary~\ref{KelmTransHGCor} implies that
$\lambda(\h)< \lambda(\mathbb{S}_{uw}(\h))$, which contradicts the maximality of $\h$ in $\mathfrak{Ca}_k(m,t)$.
If $t\geq 2$, by (A2) and (A3),
then there exists a cycle $\widehat{\mathcal{C}}\neq \mathcal{C}$ of length $2$ containing $u$.
In this case,
$e \cap g=\emptyset$ for every $g\in E(\widehat{\mathcal{C}})$ and $\mathbb{S}_{uw}(e)\notin E_\h(u)$, so Corollary~\ref{KelmTransHGCor} implies that
$\lambda(\h)< \lambda(\mathbb{S}_{uw}(\h))$, which contradicts the maximality of $\h$ in $\mathfrak{Ca}_k(m,t)$.
Hence, $w=u$, which implies that every pendant edge of $\h$ contains $u$.
The proof is completed.
\end{proof}

%%%%%%%%%%%%%%%%%%%%%%%%%%%%%%%%%%%%%%%%%%%%%%%%%%%%%%%%%%%%%%%%
%%%%%%%%%%%%%%%%%%%%%%%%%%%%%%%%%%%%%%%%%%%%%%%%%%%%%%%%%%%%%%%%

Our next goal is to prove that $\mathcal{L}^{(k)}_{m,t}$ is the unique maximizing $k$-graph in $\mathfrak{LCa}_k(m,t)$.
We start with the simplest case of $k=2$,
which will be used in the proof for $k\geq 3$.

\begin{theorem}\label{MainresutCactus222}
Let $m\geq 4$ and $t\geq 0$ be integers.
The graph $\mathcal{L}^{(2)}_{m,t}$ % illustrated in Figure~\ref{figure3}
is the unique maximizing graph in $\mathfrak{LCa}_2(m,t)$.
\end{theorem}
\begin{proof}
The proof is similar to that of Theorem~\ref{MainresutCactus111}, so we leave it to the reader.
\end{proof}

\begin{remark}
For $k\geq 3$, the proof strategy of Theorem~\ref{MainresutCactus111} cannot be directly used to determine the maximizing graph in $\mathfrak{LCa}_k(m,t)$.
The main reason is that for a $k$-graph $\mathcal{H}$ which is a linear cycle $(v_0,e_1,v_1,e_2,v_2,e_3,v_0)$ of length $3$, the $k$-graph $\mathbb{S}_{v_1v_0}(\mathcal{H})$ is not linear.
However, for $k=2$, the $2$-graph $\mathbb{S}_{v_1v_0}(\mathcal{H})$ is still a cycle of length $3$.
This is why the proof strategy of Theorem~\ref{MainresutCactus222} coincides with that of Theorem~\ref{MainresutCactus111}.
\end{remark}

%%%%%%%%%%%%%%%%%%%%%%%%%%%%%%%%%%%%%%%%%%%%%%%%%%%%%%%%%%%%%%%%
%%%%%%%%%%%%%%%%%%%%%%%%%%%%%%%%%%%%%%%%%%%%%%%%%%%%%%%%%%%%%%%%

To determine the maximizing graph in $\mathfrak{LCa}_k(m,t)$ for $k\geq 3$, we need the concept of $k$-th power of a graph.
Let $G=(V,E)$ be a graph.
For an integer $k\geq3$,
the {\it $k$-th power of $G$}, denoted by $G^{(k)}$,
is the $k$-graph with the vertex set $V\cup{\{v_{e,1},\ldots,v_{e,k-2}:e\in E}\}$ and edge set
${\{e\cup{{\{v_{e,1},\ldots,v_{e,k-2}}}\}:e\in E}\}$.
The following result establishes an identity between
$\mu(G,x)$ and $\mu(G^{(k)},x)$.

\begin{lemma}\label{largestmrootpowergraph}
Let $G$ be a graph.
Then
\[
\mu(G^{(k)},x)=x^{(1-\frac{k}{2})|V(G)|+(k-2)|E(G)|}
\mu(G,x^{\frac{k}{2}}).
\]
In particular, we have $\lambda(G^{(k)})=\lambda(G)^{\frac{2}{k}}$.
\end{lemma}
\begin{proof}
By the definition of $G^{(k)}$, we get that $p(G^{(k)},i)=p(G,i)$ for every $i$, and
$$|V(G^{(k)})|=|V(G)| + (k-2)|E(G)|.$$
The statement follows from the following calculation
\begin{align*}
\mu(G^{(k)},x)
&=\sum_{i\geq 0}(-1)^ip(G^{(k)},i)x^{|V(G^{(k)})|-ki}\\
&=\sum_{i\geq 0}(-1)^ip(G,i)x^{|V(G)|+(k-2)|E(G)|-ki}\\
&=x^{(1-\frac{k}{2})|V(G)|+(k-2)|E(G)|}
\sum_{i\geq 0}(-1)^ip(G,i)x^{\frac{k}{2}(|V(G)|-2i)}\\
&=x^{(1-\frac{k}{2})|V(G)|+(k-2)|E(G)|}\mu(G,x^{\frac{k}{2}}). \qedhere
\end{align*}
\end{proof}

With the help of Theorem~\ref{MainresutCactus222} and Lemma~\ref{largestmrootpowergraph}, the following result determines the unique maximizing $k$-graph in $\mathfrak{LCa}_k(m,t)$.

\begin{theorem}\label{MainresutCactus333}
Let $k\geq 2,$ $m\geq 4$ and $t\geq 0$ be integers.
The $k$-graph $\mathcal{L}^{(k)}_{m,t}$ % illustrated in Figure~\ref{figure3}
is the unique maximizing $k$-graph in $\mathfrak{LCa}_k(m,t)$.
\end{theorem}
\begin{proof}
If $t=0$,
then the statement follows from Lemma~\ref{pendanttreelemma}.
Assume that $t\geq 1$.
The case of $k=2$ follows from Theorem~\ref{MainresutCactus222}.
Now we suppose that $k\geq 3$ and let $\h$ be a maximizing $k$-graph in $\mathfrak{LCa}_k(m,t)$.
By Lemma~\ref{largestmrootpowergraph} and Theorem~\ref{MainresutCactus222},
it suffices to prove that $\h$ is a $k$-th power of a graph. Equivalently, every edge of $\h$ contains at least $k-2$ degree-one vertices.
We prove it by showing that $\h$ satisfies the following properties.

\begin{enumerate}[label={\rm (\alph*)}]%

\item[(B1)] Every edge of $\h$ lies in either a cycle or a pendant tree.

\item[(B2)] Every cycle of $\h$ has length $3$.

\item[(B3)] If two cycles $\mathcal{C}_1$ and $\mathcal{C}_2$ of $\h$ have a common vertex $u$,
     then $d_{\mathcal{C}_1}(u)=d_{\mathcal{C}_2}(u)=2$.

\item[(B4)] If $e$ is a pendant edge of $\h$, then there exists a cycle $\mathcal{C}$ of $\h$ and a vertex $w\in e$ such that $d_\mathcal{C}(w)=2$.

\end{enumerate}
Clearly, (B1), (B2), and (B3) follow from Lemma~\ref{structurallma}.
It remains to prove (B4).
By Lemma~\ref{structurallma}(a), if $e$ is a pendant edge of $\h$, then it is contained in a $k$-tree of $\h$, say $\T$.
Lemma~\ref{pendanttreelemma} implies that $\T$ is a $k$-star
and the center $w$ of $\T$ is contained in a cycle $\mathcal{C}$ of $\h$ by Lemma~\ref{structurallma}(a).
Moreover, Lemma~\ref{structurallma}(b) implies that $d_\mathcal{C}(w)=2$.
The proof is completed.
\end{proof}

We are now ready to prove the main result of this paper.

\begin{proof}[\proofname{ of \bf Theorem~\ref{MainresutCactus}.}]
For $m\geq 4$, the statement follows from Theorems~\ref{MainresutCactus111} and~\ref{MainresutCactus333},
respectively.
For $m=1$ or $m=2$, the statement is obvious.
If $m=3$, then $t=0$ or $t=1$.
For $(m,t)=(3,0)$, all elements of $\mathfrak{Ca}_k(3,0)$ and $\mathfrak{LCa}_k(3,0)$ are $k$-trees, so the statement follows from Lemma~\ref{pendanttreelemma}.
It remains to consider the case of $(m,t)=(3,1)$.
Clearly, $\mathfrak{LCa}_k(3,1)$ contains only one $k$-graph $\mathcal{L}^{(k)}_{3,1}$.
Also, $\mathfrak{Ca}_k(3,1)=\{\h^{(k)}_{3,1}, \mathcal{L}^{(k)}_{3,1}, \g\}$,
where $\g$ is the $k$-cactus that consists of a cycle of length $2$ with one pendant edge attached at a degree-one vertex of the cycle.
One may calculate that
$$\mu(\g,x)=x^{3k-3}-3x^{2k-3}+x^{k-3}$$
and
$$\mu(\h^{(k)}_{3,1},x)=\mu(\mathcal{L}^{(k)}_{3,1},x)
=x^{3k-3}-3x^{2k-3},$$
so
$$\left(\frac{3+\sqrt{5}}{2}\right)^{1/k}=\lambda(\g)<\lambda(\h^{(k)}_{3,1})
=\lambda(\mathcal{L}^{(k)}_{3,1})=3^{1/k}.$$
The proof is completed.
\end{proof}

\section{Concluding Remarks}\label{Concluding}

In this paper, we determine all $k$-graphs
whose largest matching root attains the maximum
among all $k$-cacti and linear $k$-cacti
with a given number of cycles and edges.
Our method combines the spectral method and the shifting operation of $k$-graphs.
We believe that there are some potential applications of Theorem~\ref{KelmTranshypergraph} and Lemma~\ref{treejointgraphslemma}.
For instance, as a quick application,
we next consider the maximizing $k$-graphs among all $n$-vertex $k$-graphs with a given number of pendant edges.

Let $\mathcal{K}_n^{(k)}$ be the complete $k$-graph of order $n$, that is, the $k$-graph with $n$ vertices and its
edge set consists of all $k$-element subsets of the vertex set.
For integers $n,k,p$, write
$$\mathcal{K}_{n,p}^{(k)}= \mathcal{K}_n^{(k)}(u)\ast \mathcal{S}_p^{(k)}(v),$$
where $u\in V(\mathcal{K}_n^{(k)})$ and $v$ is the center of the $k$-star $\mathcal{S}^{(k)}_p$.
Observe that the $k$-graph $\mathcal{K}_{n,p}^{(k)}$ is unique up to isomorphism.
The following result extends Theorem 3.3 in \cite{ZhangCY} from graphs to $k$-graphs.

\begin{theorem}\label{thmmaxpendentedges}
Let $k\geq 2$, $p\geq 1$, and  $n\geq p(k-1)+k+1$ be integers.
Then the $k$-graph $\mathcal{K}_{n-p(k-1),p}^{(k)}$ is the unique   $k$-graph whose largest matching root
attains the maximum among all $n$-vertex $k$-graphs with $p$ pendant edges.
\end{theorem}
\begin{proof}
Let $\h$ be a maximizing $k$-graph among all $n$-vertex $k$-graphs with $p$ pendant edges.
Let $e_1,\ldots, e_p$ be all pendant edges of $\h$,
and let $W$ be the $p(k-1)$ degree-one vertices of these pendant edges.
Denote by $\h-W$ the induced subgraph of $\h$ induced by the vertex set $V(\h)\setminus W$, that is, $V(\h-W)=V(\h)\setminus W$ and $E(\h-W)=\{e\in E(\h):e \subseteq V(\h)\setminus W\}$.
We claim that $\h-W$ is a complete $k$-graph of order $n-p(k-1)$.
If not, we add all possible edges to $\h-W$ until it becomes a complete $k$-graph.
Observe that the resultant $k$-graph $\widehat{\h}$ still has $p$ pendant edges and Lemma~\ref{subgraphlagroot} implies that $\lambda(\h)< \lambda(\widehat{\h})$, a contradiction to the maximality of $\h$. The claim follows.

For all $i=1,\ldots,p$,
let $v_i\in e_i$ be the root of $e_i$ in $\h$, that is, $v_i$ is the unique vertex of $e_i$ with degree at least two.
If $v_s \neq v_t$ for some $1\leq s,t\leq p$,
then we consider the $k$-graph $\mathbb{S}_{v_sv_t}(\h)$.
As we claimed before, $\h-W$ is a complete
$k$-graph of order $n-p(k-1)$.
One may check that $\mathbb{S}_{v_sv_t}(\h)-W$ is also a complete $k$-graph of order $n-p(k-1)$, and
$\mathbb{S}_{v_sv_t}(\h)$ has $p$ pendant edges and is not isomorphic to $\h$.
Moreover,
$e_s\cap e_t=\emptyset$ and $\mathbb{S}_{v_sv_t}(e_t)\notin E_\h(v_s)$, so Corollary~\ref{KelmTransHGCor} implies that $\lambda(\mathcal{H})<\lambda(\mathbb{S}_{v_sv_t}(\h)),$
which contradicts the maximality of $\h$.
Thus, $v_i = v_j$ for all $1\leq i,j\leq p$, and the proof is  completed.
\end{proof}

\end{document}